\documentclass[journal]{IEEEtran}
\usepackage{times}
\usepackage{epsfig}
\usepackage{graphicx}
\usepackage{amsmath}
\usepackage{subfigure}
\usepackage[ruled,vlined]{algorithm2e}
\usepackage{amssymb}
\usepackage{multicol}
\usepackage{multirow}
\usepackage{enumitem}
\usepackage{stfloats}
\usepackage{amsthm}
\usepackage{cite}
\usepackage{color}
\usepackage{url}
\usepackage{relsize}
\usepackage{booktabs}
\usepackage{mathtools}
\usepackage{cuted}

\usepackage{url}

\theoremstyle{plain}
\newtheorem{theorem}{Theorem}
\newtheorem{lemma}{Lemma}

\theoremstyle{definition}

\theoremstyle{remark}

\usepackage{bm}
\DeclareMathAlphabet\mathbfcal{OMS}{cmsy}{b}{n}
\newcommand{\ten}[1]{\mathbfcal{#1}} %mathcal
\newcommand{\mat}[1]{\mathbf{#1}}

\newcommand{\ccf}[1]{\textcolor{black}{#1}}

\newcommand{\thickhline}{\noalign{\hrule height 1.0pt}}
\newcommand{\parm}{{\xi}}
\newcommand{\vecpar}{\boldsymbol{\parm}}

\newcommand{\parNum}{d}

\newcommand{\out}{y}

\newcommand{\multiGPC}{\Psi }

\newcommand{\polyInd}{\alpha}
\newcommand{\basisInd}{\boldsymbol{\polyInd}}

\newcommand{\pcOrder}{p}

\newcommand{\yPC}{\sum\limits_{|\basisInd|=0}^{\pcOrder} {c_{\basisInd}  \multiGPC_{\basisInd}  (\vecpar)} }

\newcommand{\parInd}{k}

\newcommand{\muvec}{\boldsymbol{\mu} }
\newcommand{\Sigmamat}{\mathbf{\Sigma}}

\begin{document}

\title{Stochastic Collocation with Non-Gaussian Correlated Process Variations: Theory, Algorithms and Applications}

\author{Chunfeng~Cui, and~Zheng~Zhang,~\IEEEmembership{Member,~IEEE}
\thanks{Some preliminary results of this work were reported in~\cite{Cui:EPEPS2018} and received the best conference paper award at EPEPS 2018. This work was partly supported by NSF-CCF Award No. 1763699, the UCSB start-up grant and a Samsung Gift Funding.}
\thanks{Chufeng Cui and Zheng Zhang are with the Department of Electrical and Computer Engineering, University of California, Santa Barbara, CA 93106, USA (e-mail: chunfengcui@ucsb.edu, zhengzhang@ece.ucsb.edu).}
%\thanks{Max Gershman is the Department of Mathematics/Statistics and Applied Probability, University of California, Santa Barbara, CA 93106, USA (e-mail: maxwellgershman@umail.ucsb.edu).}
%\thanks{Copyright (c) 2015 IEEE. Personal use of this material is permitted. However, permission to use this material for any other purposes must be obtained from the IEEE by sending an email to pubs-permissions@ieee.org.}
}

 \maketitle

\begin{abstract}
%\boldmath
Stochastic spectral methods have achieved great success in the uncertainty quantification of many engineering problems, including electronic and photonic integrated circuits influenced by fabrication process variations. Existing techniques employ a generalized polynomial-chaos expansion, and they almost always assume that all random parameters are mutually independent or Gaussian correlated. However, this assumption is rarely true in real applications. How to handle non-Gaussian correlated random parameters is a long-standing and fundamental challenge. A main bottleneck is the lack of theory and computational methods to perform a projection step in a correlated uncertain parameter space. This paper presents an optimization-based approach to automatically determinate the quadrature nodes and weights required in a projection step, and develops an efficient stochastic collocation algorithm for systems with non-Gaussian correlated parameters. We also provide some theoretical proofs for the complexity and error bound of our proposed method. Numerical experiments on synthetic, electronic and photonic integrated circuit examples show the nearly exponential convergence rate and excellent efficiency of our proposed approach. Many other challenging uncertainty-related problems can be further solved based on this work.

\end{abstract}


\IEEEpeerreviewmaketitle

\section{Introduction}
\IEEEPARstart{P}{rocess} variation (e.g., random doping fluctuations and line edge roughness) is a major concern in nano-scale fabrications~\cite{variation2008}: even a random difference on the atomic scale can have a large impact on the electrical properties of electronic integrated circuits (IC)~\cite{Miranda}, causing significant performance degradation and yield reduction. This issue is more severe in photonic IC~\cite{selvaraja2010subnanometer,Zortman:10,chrostowski2014impact}, as photonic IC is much more sensitive to geometric variations such as surface roughness due to its large device dimension compared with the small operation wavelength~\cite{lu2017performance}. In order to address this long-standing and increasingly important issue, efficient uncertainty quantification tools should be developed to predict and control the uncertainties of chip performance under various process variations. Due to its ease of implementation, Monte Carlo~\cite{MCintro, lu2017performance} has been used in many commercial design automation tools. However, a Monte Carlo method often requires a huge number of device- or circuit-level simulation samples to achieve acceptable accuracy, and thus it is very time-consuming. As an alternative,
stochastic spectral methods~\cite{book:Dxiu} may achieve orders-of-magnitude speedup over Monte Carlo methods in many application domains.

A stochastic spectral method approximates an unknown uncertain quantity (e.g., the nodal voltage, branch current or power dissipation of a circuit) as a linear combination of some specialized basis functions such as the generalized polynomial chaos~\cite{gPC2002}.
Both intrusive (i.e., non-sampling) solvers (e.g., stochastic Galerkin~\cite{sfem} and stochastic testing~\cite{zzhang:tcad2013}) and non-intrusive (i.e., sampling) solvers (e.g., stochastic collocation~\cite{col:2005}) have been developed to compute the unknown weights of these pre-defined basis functions. These techniques have been successfully applied in electronic IC
 \cite{vrudhula2006hermite, pham2014decoupled, Stievano:2011_1,  Strunz:2008,Pulch:2011_1, Rufuie2014, manfredi:tcas2014, zzhang:tcas2_2013,ahadi2016sparse,yucel2015me}, MEMS~\cite{zzhang_cicc2014, zzhang:huq_tcad} and photonic IC~\cite{twweng:optsEx, waqas2018stochastic} applications, achieving orders-of-magnitude speedup than Monte Carlo when the number of random parameters is small or medium. In the past few years, there has been a rapid progress in developing high-dimensional uncertainty quantification solvers. Representative results include tensor recovery~\cite{zhang2017big}, compressive sensing \cite{li2010finding}, ANOVA (analysis of variance) or HDMR (high-dimensional model representation)~\cite{yang2012adaptive,ma2010adaptive,zzhang_cicc2014}, matrix low-rank approximation~\cite{el2010stochastic}, stochastic model order reduction~\cite{el2010variation}, and hierarchical uncertainty quantification~\cite{zzhang_cicc2014,zzhang:huq_tcad}.

The above existing techniques use generalized polynomial chaos~\cite{gPC2002} as their basis functions, and they assume that all process variations can be described by independent random parameters. Unfortunately, this assumption is not true in many practical cases. For instance, the geometric or electrical parameters influenced by the same fabrication step are often highly correlated. In a system-level analysis, the performance parameters from circuit-level simulations are used as the inputs of a system-level simulator, and these circuit-level performance quantities usually depend on each other due to the network coupling and feedback. In photonic IC, spatial correlations may have to be considered for almost all components due to the small wavelength~\cite{pond2017predicting}. All these correlations are not guaranteed to be Gaussian, and they can not be handled by pre-processing techniques such as principal component analysis~\cite{chang2003statistical}. Karhunen-Loe\`ve theorem~\cite{bhardwaj2006modeling, sapatnekar2011overcoming} and Rosenblatt transformation \cite{rosenblatt1952remarks} may transform correlated parameters into uncorrelated ones, but they are error-prone and not scalable.

This paper develops new theory and algorithms of uncertainty quantification with non-Gaussian correlated process variations. Two main challenges arise when we quantify the impact of correlated non-Gaussian process variations. Firstly, we need to develop a new set of stochastic basis functions to capture the effects of non-Gaussian correlated process variations. Soize~\cite{soize2004physical} suggested to modify the generalized polynomial chaos, but the resulting non-polynomial basis functions are non-smooth and numerically unstable. Secondly, we need to develop a spectral method (either stochastic collocation or stochastic Galerkin) to compute the weights of the new basis functions. This requires performing a projection step by an accurate numerical integration in a multi-dimensional correlated parameter space. While the numerical integration in a one-dimensional space~\cite{Golub:1969} or a two-dimensional correlated square space~\cite{xu2018optimal} is well-studied, accurate numerical integration in a higher-dimensional correlated parameter space remains a challenge. During the preparation of this manuscript, the authors noticed some recent results on stochastic Galerkin \cite{paulson2017arbitrary, navarro2014polynomial} and sensitivity analysis for dependent random parameters~\cite{liu2018data}. However, the theoretical analysis and numerical implementation of stochastic collocation have not been investigated for systems with non-Gaussian correlated parameters.

% Motivated by~\cite{ryu2015extensions, keshavarzzadeh2018numerical}, we propose an optimization solver for general multidimensional numerical integration.

\textbf{Main Contributions.} This paper presents a novel stochastic collocation approach for systems with correlated non-Gaussian uncertain parameters. Our main contributions include:
\begin{itemize}
\item The development of a set of basis functions that can capture the impact of non-Gaussian correlated process variations. Some numerical implementation techniques are also presented to speed up the computation.
\item An optimization-based quadrature rule to perform projection in a multi-dimensional correlated parameter space.  Previous stochastic spectral methods use id~\cite{nobile2008sparse} or Gauss quadrature~\cite{Golub:1969}, which is not applicable for non-Gaussian correlated cases. We reformulate the numerical quadrature problem as a nonlinear optimization problem, and apply a block coordinate descent method to solve it. Our approach can automatically determinate the number of quadrature samples. We also provide a theoretical analysis for the the upper and lower bounds of the number of quadrature samples required in our framework.

\item Theoretical error bound of our algorithm. We show that: (1) we can obtain the exact solution under some mild conditions when the stochastic solution is a polynomial function; (2) for a general smooth stochastic solution, an upper error bound exists for our stochastic collocation algorithm, and it depends on the distance of the unknown solution to a polynomial set as well as the numerical error of our optimization-based quadrature rule.
\item A set of numerical experiments on synthetic and realistic electronic and photonic IC examples. The results show the fast convergence rate of our method and its orders-of-magnitude (700$\times$ to 6000$\times$) speedup than Monte Carlo.
\end{itemize}

Before discussing about the technical details, we summarize some of the frequently used notations in Table~\ref{tab:notations}.

\begin{table}
\label{tab:notations}
\caption{Notations in this paper}
\begin{center}
\begin{tabular}{rl}
\hline
$d$ & number of random parameters describing process variations\\
$p$ & the highest total polynomial order\\
$M$ & number of quadrature nodes\\
$\vecpar$ &   a vector denoting $d$ uncertain parameters  \\
$\rho(\vecpar)$ & the joint probability density function of $\vecpar$\\
$\vecpar_k$ & the value of $\vecpar$ at a quadrature node \\
%$\bar \vecpar$ & all sampling nodes $[\vecpar_1;\ldots;\vecpar_M]\in\mathbb{R}^{dM}$ \\
$w_k$ &  nonnegative weight associated with $\vecpar_k$ \\
$\mat{w}$ &  $d$-dimensional vector of $w_k$ \\
$\mat{e}_1$ & a vector of the form $[1,0,\ldots,0]^T$\\
$\basisInd$ & $d$-dimensional vector indicating order of a multivariate polynomial\\
$\multiGPC_{\basisInd}(\vecpar)$ & orthonormal basis functions with $\mathbb{E}[\multiGPC_{\basisInd}(\vecpar)\multiGPC_{\boldsymbol{\beta}}(\vecpar)]=\delta_{\basisInd\boldsymbol{\beta}}$ \\
$c_{\basisInd}$ & coefficient or weight of $\multiGPC_{\basisInd}(\vecpar)$ in the expansion\\
$\tilde c_{\basisInd}$ & approximation for $c_{\basisInd}$ by numerical integration\\
$\multiGPC_{j}(\vecpar)$ & $\multiGPC_{\basisInd}(\vecpar)$ in the graded lexicographic order\\
$c_j$ & coefficient or weight of $\multiGPC_{j}(\vecpar)$ in the expansion\\
$\tilde c_j$ & approximation for $c_j$ by numerical integration\\
$\out(\vecpar)$ & the unknown stochastic solution to be computed \\
$\tilde{\out}(\vecpar)$ & approximation of $\out(\vecpar)$ by our method\\
$\out_p(\vecpar)$ & the projection of $\out(\vecpar)$ onto  polynomial set $\mathcal{S}_p$ \\
$\mathcal{S}_p$ & the set of $d$-dimensional polynomials with total order $\le p$\\
$N_p$ & the number of  $d$-dimensional monomials with order $\le p$\\
\hline
\end{tabular}
\end{center}
\end{table}

%% section 2 : Preliminaries
\section{Preliminaries}

%% Stochastic collocation
\subsection{Review of Stochastic Collocation}
% Stochastic collocation methods are used for the numerical solution of PDE's, ODE's and Integral Equations. Using polynomial functions, as well as nodes and weights defined from quadrature rules, stochastic collocation can be used to find a solution for the coefficients. While previous collocation methods such as Stochastic Galerkin requires multidimensional summations over the basic functions, the stochastic collocation method collapse those summations to a 1-Dimensional summation~\cite{mathelin2003stochastic}. Stochastic Collocation methods do not use randomly chosen samples, but use a deterministic grid of nodes which can be solved using hierarchical basis functions~\cite{teckentrup2015multilevel}. Stochastic Collocation has been known to be affected by the curse of dimensionality, yet by using a deterministic approach such as quadrature rules accurate results can be produced for a problem with a medium number of inputs.

Stochastic collocation \cite{xiu2005high, Ivo:2007, Nobile:2008, Nobile:2008_2} is the most popular non-intrusive stochastic spectral method. The key idea is to approximate the unknown stochastic solution as a linear combination of some specialized basis functions, and to compute the weights of all basis functions based on a post-processing step such as projection.
In order to implement the projection, one needs to do some device- or circuit-level simulations repeatedly for some parameter samples selected by a quadrature rule. Given a good set of basis functions and an accurate quadrature rule, stochastic collocation may obtain a highly accurate result with only a few repeated simulations and can achieve orders-of-magnitude speedup than Monte Carlo when the number of random parameters is small or medium.

Specifically, let $\vecpar=[\parm_1, \cdots, \parm_{\parNum}]^T \in \mathbb{R}^{\parNum}$ denotes a set of random parameters that describe some process variations. We aim to estimate the uncertainty of $\out (\vecpar)$, which is a parameter-dependent output of interest such as the power dissipation of a memory cell, the 3-dB band width of an amplifier or the frequency of an oscillator. In almost all chip design cases, we do not have a closed-form expression of $y(\vecpar)$, and we have to call a time-consuming device- or circuit-level simulator (which involves solving large-scale differential equations) to obtain the numerical value of $y(\vecpar)$ for each specified sample of $\vecpar$.
%Generalized polynomial-chaos expansion can be applied
Stochastic spectral methods aim to approximate $\out(\vecpar)$ via
\begin{equation}
\label{eq:ygpc}
\out (\vecpar) \approx \yPC, \; {\rm with}\; \mathbb{E}\left[{\multiGPC}_{\basisInd}  (\vecpar)\multiGPC_{\boldsymbol{\beta }}\left( \vecpar \right)\right ]=\delta_{\basisInd, \boldsymbol{\beta }}.
\end{equation}
Here $\mathbb{E}$ denotes the expectation operator, $\delta$ denotes a Delta function, the basis functions $\{{\multiGPC}_{\basisInd} \left(\vecpar\right)\}$ are some orthonormal basis functions indexed by a vector $\basisInd=[\alpha_1,\cdots, \alpha_{\parNum}] \in \mathbb{N}^{\parNum}$. The total order of the basis function $|\basisInd|=\alpha_1+\ldots+\alpha_d$ is bounded by $p$, and thus the total number of basis functions is
\begin{equation}
N_p=\binom{p+d}{d}=(p+d)!/(p!d!).
\end{equation}
The coefficient $c_{\basisInd}$ can be obtained by a projection
\begin{equation}
\label{eq:yproject}
c_{\basisInd}= \mathbb{E}\left[ \out (\vecpar)  {\multiGPC}_{\basisInd} (\vecpar)\right]=
\int\limits_{\mathbb{R}^d}  {\out (\vecpar)  {\multiGPC}_{\basisInd} (\vecpar)\rho({\vecpar}) d\vecpar},
\end{equation}
where $\rho(\vecpar)$ is the joint probability density function.
The  integral in (\ref{eq:yproject}) needs to be evaluated with  numerical integration
\begin{equation}
\label{eq:yTP}
c_{\basisInd}\approx\sum\limits_{\parInd=1}^M {\out (\vecpar_{\parInd})  {\multiGPC}_{\basisInd}  (\vecpar_{\parInd})  w_{\parInd}  }.
\end{equation}
where $\{\vecpar_k\}_{k=1}^M$ are the quadrature nodes, and  $\{w_k\}_{k=1}^M$ are the corresponding quadrature weights. The key of stochastic collocation is to choose proper basis functions and an excellent quadrature rule, such that $M$ is as small as possible in \eqref{eq:yTP}.

\subsection{Existing Solutions for Independent Cases}

Most existing stochastic spectral methods assume that $\vecpar=[\xi_1,\ldots,\xi_d]^T$ are mutually independent. In this case, given the marginal density function $\rho_k(\xi_k)$ of each parameter, the joint density function is $\rho(\vecpar)=\Pi_{k=1}^d\rho_k(\xi_k)$.
Consequently, an excellent choice of basis functions is the generalized polynomial chaos~\cite{gPC2002}: the multivariate basis function is obtained as the product of some univariate polynomial basis functions
\begin{equation}\label{equ:basis_inde}
\multiGPC_{\basisInd}(\vecpar)=\phi_{1,\alpha_1}(\xi_1)\ldots\phi_{d,\alpha_d}(\xi_d).
\end{equation}
Here each univariate basis function $\phi_{k,\alpha_k}(\xi_k)$ can be constructed via the well-known three-term recurrence relation~\cite{gautschi1982generating}, and the univariate basis functions of the same parameter $\xi_k$ are mutually orthonormal with respect to the marginal density function $\rho_k(\xi_k)$. %In generalized polynomial chaos~\cite{gPC2002} polynomials, $\basisInd=[\alpha_1,\cdots, \alpha_{\parNum}] \in \mathbb{N}^{\parNum}$ is a vector indicating the highest polynomial order of each parameter in the corresponding basis.

When $\vecpar$ are mutually independent, the quadrature points and weights in \eqref{eq:yTP} are often constructed via the tensor product of one-dimensional quadrature points and weights. Specifically, denote $\{\xi_{i_k}, w_{i_k}\}$ as the quadrature nodes and weights for the one-dimensional parameter $\xi_{k}$ (for instance, via Gaussian quadrature rule~\cite{Golub:1969}), then $\vecpar_{i_1\ldots i_d}=[\xi_{i_1},\ldots,\xi_{i_d}]^T$
and $w_{i_1\ldots i_d}=w_{i_1}\ldots w_{i_d}$ are the quadrature points and weights for a $d$-dimensional problem.
Another popular approach is the sparse grid technique \cite{nobile2008sparse},\ccf{\cite{Gerstner:1998, Hengliang:2007, sparse_grid:2000}}, which can significantly reduce the number of quadrature points by exploiting the nested structure of the quadrature points of different accuracy levels.

 % \begin{equation}
% \label{eq:yTPind}
% c_{\basisInd}\approx\sum\limits_{i_1\ldots i_d=1}^q {\out (\vecpar_{i_1\ldots i_d})  \phi_{1,\alpha_1}(\xi_{i_1})\ldots\phi_{d,\alpha_d}(\xi_{i_d})  w_{i_1}\ldots w_{i_d}},
% \end{equation}

\subsection{Non-Gaussian Correlated Cases}
In general, $\vecpar$ can be non-Gaussian correlated, and the joint density $\rho(\vecpar)$ cannot be written as the product of the individual marginal density functions. As a result the multivariate basis function can not be obtained as in \eqref{equ:basis_inde}. It is also hard to choose a small number of quadrature nodes $\{\vecpar_{\parInd}\}$ and weights $\{w_{\parInd}\}$ that can produce highly accurate integration results.

In order to quantify the impact of non-Gaussian correlated uncertainties, Soize~\cite{soize2004physical} suggested a set of non-smooth orthonormal basis functions by modifying the generalized polynomial chaos~\cite{gPC2002}. The modified basis functions were employed in~\cite{twweng:optsEx} for the variability analysis of silicon photonic devices. However, the algorithm does not converge well due to the numerical instability of the basis functions, and
designers cannot easily extract statistical information (e.g.,
mean value and variance) from the obtained solution. In the applied math community, multivariate orthogonal
polynomials may be constructed via the multivariate three-term recurrence~\cite{xu1993multivariate,barrio2010three}. However, the theories in~\cite{xu1993multivariate, barrio2010three} either are hard to implement or can only guarantee weak orthogonality.

\begin{figure*}[t]
	\centering
		\includegraphics[width=5.0in]{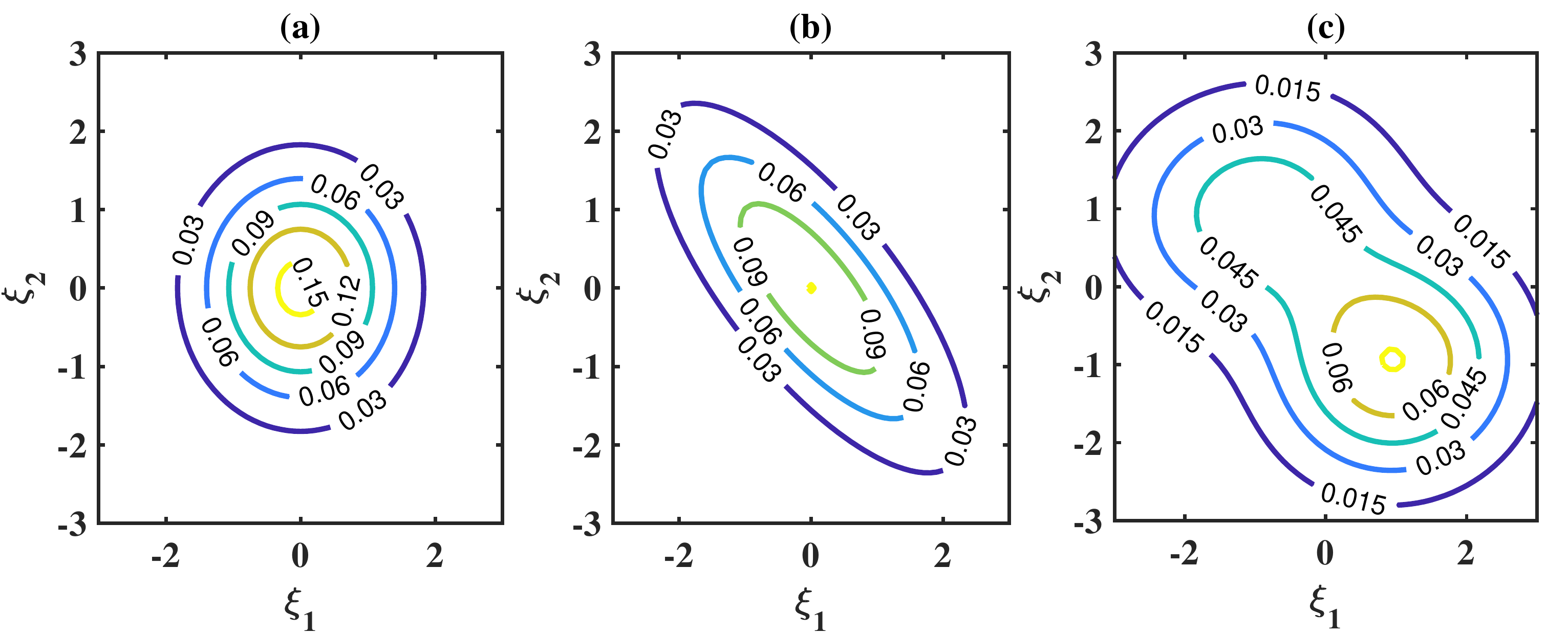}
\caption{Several joint density functions. (a) independent Gaussian; (b) correlated Gaussian; (c) correlated non-Gaussian (e.g., a Gaussian-mixture distribution).}
	\label{fig:gmdistribution}
\end{figure*}

 \section{Proposed Orthonormal Basis Functions}
 \label{sec:basisfunc}
This section presents a set of smooth orthonormal basis functions that can capture the impact of non-Gaussian correlated random parameters. The proposed basis functions allow us to approximate a smooth $y(\vecpar)$ with a high accuracy and to extract its statistical moments analytically or semi-analytically.
%  For one dimensional case, the orthogonal basis function can be calculated via the   three-term recurrence relation~\cite{Walter:1982}.

%% Gram-Schmidt method
\subsection{Generating Multivariate Orthonormal Polynomials}
We adopt a Gram-Schmidt approach to calculate the basis functions recursively. The Gram-Schmidt method was used for vector orthogonalization in the Euclidean space~\cite{golub2012matrix}.
%  See Fig.~\ref{fig:gram} for an illustration of orthogonalization process of $u$ and $v$.
It can also be generalized to construct some orthogonal polynomial functions. The key difference here is to replace the vector inner product with the functional expectations.

%   \begin{figure}[t]
% 	\centering \includegraphics[width=2.0in]{fig/Gram_schmidt}
% \caption{Gram-Schmidt orthogonalization of  $\mat{v}$ to be orthogonal with $\mat{u}_1$.  $\mat{u}_2$ is the result vector.}
% 	\label{fig:gram}
% \end{figure}

Specifically, we first reorder the monomials $\vecpar^{\basisInd}=\xi_1^{\alpha_1}\ldots \xi_d^{\alpha_d}$ in the  graded lexicographic order, and denote them as $\{p_j(\vecpar)\}_{j=1}^{N_p}$. For instance, when $d=2$ and $p=2$, there is
 \begin{equation*} \{p_j(\xi_1,\xi_2)\}_{j=1}^6=\{1,\xi_1,\xi_2,\xi_1^2,\xi_1\xi_2,\xi_2^2\}.
 \end{equation*}
Then we set $\multiGPC_1(\vecpar)   = 1$ and generate orthonormal polynomials $\{\multiGPC_j(\vecpar)\}_{j=1}^{N_p}$ in the correlated parameter space recursively by
 \begin{align} &\hat{\multiGPC}_j(\vecpar) = p_j(\vecpar)-\sum_{i=1}^{j-1} \mathbb{E}[ p_j(\vecpar)\multiGPC_i(\vecpar)] \multiGPC_i(\vecpar),%\text{proj}_{\multiGPC_i}(p_j(\vecpar)),
 \\
 &\multiGPC_j(\vecpar)  = \frac{\hat{\multiGPC}_j(\vecpar)}{\sqrt{\mathbb{E}[\hat{\multiGPC}^2_j(\vecpar)]}},\ j=2,\ldots,N_p.\label{equ:normal}
 \end{align}
 The basis functions defined by this approach are unique under the specific order of  monomials. If the ordering of monomials are changed, one can get another set of basis functions. Since the basis functions are orthonormal polynomials, we can easily extract the mean value and statistical moment of an approximated stochastic solution.

Note that recently we have also proposed a set of orthogonal polynomial basis function via a Cholesky decomposition~\cite{Cui2018}. The method in~\cite{Cui2018} is easy to implement and suitable for high-dimensional cases, but the resulting basis functions can be occasionally inaccurate due to the numerical instability of the Cholesky factorization on a large ill-conditioned covariance matrix. This paper focuses on the fundamental problems of stochastic collocation for correlated cases, therefore, we generate basis functions via the Gram-Schmidt method.

\subsection{Numerical Implementation Issues}
The main challenge in the basis function generation is to compute the expectations in a correlated parameter space, which involves evaluating the moments $\mathbb{E}[\vecpar^{\basisInd}]$ up to order $2p$. Some techniques can be used to speed up the computation.

In practice, the process variations are generally described by a set of measurement data from testing chips, and their joint density function $\rho(\vecpar)$ is fitted using some density estimators. A widely used model is the Gaussian mixture:
 \begin{equation}
 \rho(\vecpar) =\sum_{i=1}^n r_i \mathcal{N}(\vecpar | \muvec_i, \Sigmamat_i), \; {\rm with}\; r_i>0 , \; \sum_{i=1}^n r_i=1.
 \end{equation}
Here $\mathcal{N}(\vecpar | \muvec_i, \Sigmamat_i)$ denotes a multi-variate Gaussian distribution with mean $\muvec_i \in\mathbb{R}^d$ and a covariance matrix $\Sigmamat_i \in \mathbb{R}^{d \times d}$. Fig.~\ref{fig:gmdistribution} compares the Gaussian mixture model with independent and correlated Gaussian distributions. With a Gaussian mixture, the moments can be computed accurately using a functional tensor train approach (see Section 3.3 of~\cite{Cui2018}).

For general cases, one may estimate the moments by changing the variables and density function:
\begin{equation}
\mathbb{E}[\vecpar^{\basisInd}]=\int\limits_{\mathbb{R}^d}  {g_{\basisInd} ({\boldsymbol{\eta}})\hat{\rho}({\boldsymbol{\eta}}) d {\boldsymbol{\eta}} }, \; {\rm with}\; g_{\basisInd} ({\boldsymbol{\eta}})=\frac{{\boldsymbol{\eta}}^{\basisInd}\rho({\boldsymbol{\eta}})}{\hat{\rho}({\boldsymbol{\eta}})}.
\end{equation}
Here $\hat{\rho}({\boldsymbol{\eta}})$ denotes the joint density function of independent random parameters ${\boldsymbol{\eta}} \in \mathbb{R}^d$. Then, standard quadrature methods such as sparse grid~\cite{nobile2008sparse} or tensor-product Gauss quadrature can be used to evaluate the integration. The tensor-train-based method in~\cite{zzhang:huq_tcad} can be used to reduce the integration cost when $d$ is large. The potential limitation is that it may be non-trivial to obtain highly accurate results if $g_{\basisInd} ({\boldsymbol{\eta}})$ is highly nonlinear or even non-smooth. Note that we only need to use a high-order quadrature rule in an independent parameter space and repeatedly evaluate some cheap closed-form functions here, and we do not need to perform expensive device or circuit simulations when we compute the basis functions. %Therefore, we can use a high-order quadrature method to evaluate the moments with good accuracy in this step.

In this paper, we use Gaussian mixture models to describe non-Gaussian correlated uncertainties, and we employ the functional tensor-train method~\cite{Cui2018} for moment computation.

% Assume that the joint density function follows the  Gaussian-mixture distribution.
% The key ingredient in this approach is to transform the correlated parameters $\vecpar$ into independent ones $\boldsymbol{\eta}$ via a linear transformation, and to recursively calculate the higher order moments via  the Hadamard product of  two lower order tensor train decompositions \cite{oseledets2011tensor}.  Specifically, if we have computed the functional tensor train decomposition of two low dimensional monomials, $\vecpar^{\basisInd_1}=\mathbf{E}_0(\boldsymbol{\mu})\mathbf{E}_1(\eta_1)\cdots \mathbf{E}_{d}(\eta_d)$ and $\vecpar^{\basisInd_2}=\mathbf{F}_0(\boldsymbol{\mu})\mathbf{F}_1(\eta_1)\cdots \mathbf{F}_{d}(\eta_d)$,
% then the tensor train decomposition for $\vecpar^{\basisInd} = \vecpar^{\basisInd_1}\cdot \vecpar^{\basisInd_2}$ follows directly as
%  \begin{align}
% \vecpar^{\basisInd}& =
%  \mat{G}_0(\boldsymbol{\mu})\mat{G}_1(\eta_1)\cdots \mat{G}_{d}(\eta_d),\; \rm{with \ }\mat{G}_i = \mathbf{E}_i \otimes \mathbf{F}_i.
%  \label{equ:tt_higherorder}
% \end{align}
%  As a result, we can construct the tensor train decompositions recursively.
% Please refer to Section-III of \cite{Cui2018} for more details.

 %% section 4: optimization solver
 \section{Optimization-Based Quadrature}

After constructing the basis functions, we still need to choose a small number of the quadrature nodes and weights in order to calculate $c_{\basisInd}$ by (\ref{eq:yTP}) with a small number of device- or circuit-level simulations.
%  The quadrature nodes and weights for independent parameters can be determined by solving an eigenvalue problem resulting from a three-term recurrence relation~\cite{Golub:1969}. However, the method in~\cite{Golub:1969} does not work for non-Gaussian correlated random parameters.
Motivated by \cite{ryu2015extensions, keshavarzzadeh2018numerical}, we present an optimization model to decide a proper quadrature rule. Our method differs from \cite{ryu2015extensions, keshavarzzadeh2018numerical} in both algorithm framework and theoretical analysis. Firstly, while~\cite{ryu2015extensions} only updates the quadrature weights by linear programing, we optimize the quadrature samples and weights by nonlinear optimization. Secondly, our optimization setup differs from that in~\cite{keshavarzzadeh2018numerical}: we minimize the integration error of our proposed multivariate orthonormal basis functions, such that the resulting quadrature rule is suitable for quantifying the impact of non-Gaussian correlated uncertainties. Thirdly, we handle the nonnegative constraint of the weight $\mat{w}$ and the nonlinear objective function of $\bar \vecpar$ separately via a block coordinate descent approach.
Fourthly, we propose a novel initializing method via weighted complete linkage clustering. Finally, we present theoretical results regarding the algorithm complexity and error bound.
Our method is summarized in Algorithm~\ref{alg:update_coeff}, and we elaborate the key ideas below.
\begin{algorithm}[t]
\label{alg:update_coeff}
\caption{Proposed stochastic collocation method}
 %\SetAlgoNoLine
      \SetKwInput{Input}{Input}
      \SetKwInput{Output}{Output}
% \Input{Initial  point  number $M_0$, the maximal point number $M_{\max}$ and minimal point number $M_{\min}$.}
 \begin{itemize}
  \item[Step 1] Initialize the quadrature nodes and weights via Algorithm~\ref{alg:cluster}.  %according to Section \ref{section:initialNodes}.
 \item[Step 2] \textbf{Increase phase.} Update the quadrature nodes and weights by solving~\eqref{equ:NLS}. If Alg.~\ref{alg:BCD} fails to converge, increase the node number and go back to Step~1.
 \item[Step 3] \textbf{Decrease phase.} Decrease the number of nodes, and update them by solving~\eqref{equ:NLS}. Repeat Step~3 until no points can be deleted \ccf{[in other words, the objective function of \eqref{equ:NLS} fails to reduce below a prescribed threshold].} Return the nodes and weights.
 \item[Step 4] Call a deterministic simulator to compute $\{\out(\vecpar_k)\}_{k=1}^M$. Then compute the coefficients $\{c_{\basisInd}\}$ via (\ref{eq:yTP}).
 \end{itemize}
 \Output{The coefficients $\{c_{\basisInd}\}$ in (\ref{eq:ygpc}).}
\end{algorithm}

 \subsection{Optimization Model of Our Quadrature Rule}
 \label{subsec:BCD}

% There is no explicit formula for calculating the optimal quadrature nodes and weights for multi-dimensional numerical integration.

Our idea is to compute a set of quadrature points and weights that can accurately estimate the numerical integration of some testing functions. \ccf{{Given} a joint density function $\rho(\vecpar)$}, we seek for the quadrature nodes and weights $\{\vecpar_k, w_k\}_{k=1}^M$ by matching the integration of basis functions up to order $2p$:
\begin{multline}
\label{equ:nmint1}
\mathbb{E}[\multiGPC_{j}(\vecpar)]=\int\limits_{\mathbb{R}^d}\multiGPC_{j}(\vecpar) \rho(\vecpar)d\vecpar=\sum_{k=1}^{M}\multiGPC_j(\vecpar_k)w_k, \\
\forall j=1,\ldots,N_{2p}.
\end{multline}
Here, $N_{2p}=\binom{2p+d}{d}$ denotes the total number of basis functions with their total order bounded by $2p$.

We choose the above testing functions based on two reasons. Firstly, it is easy to show that $\mathbb{E}[\multiGPC_{j}(\vecpar)]=\mathbb{E}[\multiGPC_{j}(\vecpar)\multiGPC_1(\vecpar)] =\delta_{1j}$. Secondly, we can show that for any polynomial function $f(\vecpar)$ bounded by order $2p$, the integration of $f(\vecpar)$ weighted by the density function $\rho(\vecpar)$ (i.e., $\mathbb{E}\left[f(\vecpar)\right ]$) can be written as the weighted sum of $\mathbb{E}[\multiGPC_{j}(\vecpar)]$'s, and therefore one can get the exact integration result if \eqref{equ:nmint1} holds. In stochastic collocation, if $\out (\vecpar)$ is a polynomial function bounded by order $p$, then $c_{\basisInd}=\mathbb{E} \left[ \out (\vecpar)  {\multiGPC}_{\basisInd} (\vecpar)\right]$ can be accurately computed for {\it every} basis function with $|\boldsymbol{\alpha}| \leq p$ if \eqref{equ:nmint1} holds. The detailed derivations are given in Theorem \ref{thm:exactrecovery} of Section~\ref{sec:theory}.

%  \begin{align}
%  \label{equ:nmint1}
%  \mathbb{E}[\multiGPC_{j}(\vecpar)]=\sum_{k=1}^{M}\multiGPC_j(\vecpar_k)  w_k=\delta_{1j},\ \forall\, j=1,\ldots,N,
%  \end{align}
% with   $\delta_{1j}=1$ if $j=1$ and $\delta_{1j}=0$ otherwise.

In practice, we propose to rewrite \eqref{equ:nmint1} as the following nonlinear least-square problem
 \begin{equation}\label{equ:NLS}
\min_{\bar{\vecpar},\mat{w}\ge0}\quad \|\mat{\Phi}(\bar{\vecpar})\mat{w}-\mat{e_1}\|_2^2,
\end{equation}
where $\bar{\vecpar}=[\vecpar_1^T,\ldots,\vecpar_M^T]^T\in\mathbb{R}^{Md}$, $\mat{w}=[w_1,\ldots,w_M]^T\in\mathbb{R}^{M}$, $\mat{e}_1=[1,0,\ldots,0]^T\in\mathbb{R}^{N_{2p}}$,  $\mat{\Phi}(\bar{\vecpar})$ is a matrix of size $N_{2p} \times M$ with the $(j,k)$-th element being $(\mat{\Phi}(\bar{\vecpar}))_{jk}=\multiGPC_j(\vecpar_k)$, \ccf{$\|\cdot\|_2$ denotes the Euclidean norm.}
Here, we also require the quadrature weights  to be nonnegative.
This requirement is an natural extension of the one-dimensional Gauss quadrature rule \cite{Golub:1969}, and it can help our theoretical analysis in Section~\ref{sec:theory}.

\subsection{A Block Coordinate Descent Solver for \eqref{equ:NLS}}

The total number of unknowns in \eqref{equ:NLS} is $M(d+1)$, which becomes large as $d$ increases. In order to improve the scalability of our algorithm, we solve \eqref{equ:NLS} by a block coordinate descent method. The idea is to update the parameters block-by-block: at the $t$-th  iteration, we firstly fix $\bar{\vecpar}^{t-1}$ and solve a $\mat{w}$-subproblem to update $\mat{w}^{t}$, then fix $\mat{w}^{t}$ and solve a $\vecpar$-subproblem to update $\bar{\vecpar}^{t}$.

 \textbf{$\mat{w}$-subproblem.} Suppose $\bar{\vecpar}^{t-1}=[\vecpar_1^{t-1};\ldots;\vecpar_M^{t-1}]$ is fixed, then (\ref{equ:NLS}) reduces to a convex linear least-square problem
  \begin{equation}\label{equ:NLSw}
\mat{w}^{t}=\arg\min_{\mat{w}\ge0}\quad \|\mat{\Phi}(\bar{\vecpar}^{t-1}) \mat{w}-\mat{e}_1\|_2^2.
 \end{equation}

 \textbf{$\vecpar$-subproblem.} When $\mat{w}^{t}$ is fixed, we apply the Gaussian Newton method to update the quadrature samples:
 \begin{equation}\label{equ:NLSd}
 \vecpar_k^{t}=\vecpar_k^{t-1}+\mat{d}_k^t, \text{ with } \{\mat{d}_k^t\}=\arg\min_{\{\mat{d}_k\}} \ \|\sum_{k=1}^M\mat{G}_k^t\mat{d}_k + \mat{r}^t\|_2^2.
 \end{equation}
 Here,  $\mat{r}^t = \mat{\Phi}(\bar{\vecpar}^{t-1})\mat{w}^{t}-\mat{e}_1 \in \mathbb{R}^{N_{2p}}$ denotes the residual, $\mat{G}_k^t\in \mathbb{R}^{N_{2p}\times d}$ is the Jacobian matrix of $\mat{r}^t$ with respect to $\vecpar_k^{t-1}$.
\ccf{ In practice, we run the step in (\ref{equ:NLSd}) once and go back to the $\mat{w}$-step. This is actually the inexact block coordinate approach~\cite{tappenden2016inexact}.}
The pseudo codes of our block coordinate descent solver are summarized in Algorithm~\ref{alg:BCD}. \ccf{ Here we use an $\ell _1$-norm in the stopping criteria since it enables us to bound the error of our whole framework in Section~\ref{sec:theory}.}

\ccf{
We note that some other approaches can also solve the non-convex optimization problem \eqref{equ:NLS}. When the number of unknown variables is small, we can obtain a globally optimal solution via the polynomial optimization solver based on a semi-definite positive relaxation~\cite{lasserre2001global}. The Levenberg-Marquardt approach or the trust region algorithm \cite{nocedal2006numerical} can also be used to solve the $\vecpar$-subproblem, but they are more expensive than our solver. Our optimization solver converges very well in practice. As will be shown in Section~\ref{sec:theory}, our stochastic collocation framework actually does not necessarily require a locally or globally optimal solution of~\eqref{equ:NLS} at all. Instead, it only requires the objective function to be sufficiently small at the obtained quadrature samples and weights.}

 \begin{algorithm}[t]
\label{alg:BCD}
\caption{Block coordinate descent solver for \eqref{equ:NLS}}
 %\SetAlgoNoLine
      \SetKwInput{Input}{Input}
      \SetKwInput{Output}{Output}
\Input{Initial quadrature nodes $\vecpar_1,\ldots,\vecpar_M$, the maximal iteration $n_{\max}$, and the tolerance $\epsilon$.}
\For{$t=1,\ldots,n_{\max}$}
{
Update the weights $\mat{w}^t$ via solving (\ref{equ:NLSw});\\
Update the nodes  $\bar{\vecpar}^t$ via solving (\ref{equ:NLSd});\\
\If{$ \|\mat{\Phi}(\bar{\vecpar}^t)\mat{w}^t-\mat{e_1}\|_1\le \epsilon$ is satisfied}{break;}
 }
\Output{Optimal nodes and weights $ \{ \vecpar_k,w_k\}_{k=1}^M$.}
\end{algorithm}
%% end of algorithm

\subsection{Initializing Quadrature Nodes and Weights}
 \label{section:initialNodes}

The nonlinear least square problem \eqref{equ:NLS} is non-convex, and generally it is hard to obtain the global optimal solution. In practice, accurate results can be obtained once we can use good initial guesses for the quadrature nodes and weights.

In Step 3 of Algorithm \ref{alg:update_coeff}, we need to find a quadrature rule with fewer nodes after some pairs of quadrature samples and weights have already been calculated. In this case, we can simply delete one node with the smallest weight, and choose all other samples and their corresponding weights as the initial condition for the subsequent optimization problem.

In Step 1 of Algorithm \ref{alg:update_coeff}, we need to generate some initial nodes from scratch. We firstly generate $M_0 \gg M$ nodes via Monte Carlo. In Monte Carlo sampling, all samples have the same weights $1/M_0$. In order to improve the convergence, we keep all samples unchanged but refine their weights by solving the w-subproblem in (\ref{equ:NLSw}).
These $M_0$ initial nodes are then grouped into $M$ clusters, and the resulting cluster centers are set as the initial samples for whole nonlinear least-square optimization problem. \ccf{ This choice of initial guess proves to work very well in practice, because Monte Carlo itself is an integration rule with statistical accuracy guarantees. }

 \begin{algorithm}[t]
\label{alg:cluster}
\caption{Weighted complete linkage clustering}
 %\SetAlgoNoLine
      \SetKwInput{Input}{Input}
      \SetKwInput{Output}{Output}
\Input{The number of cluster $M$, and $M_0=3M$ initial   nodes $\vecpar_1,\ldots,\vecpar_{M_0}$.}
Calculate the  weights for $\vecpar_1,\ldots,\vecpar_{M_0}$ by solving
\eqref{equ:NLSw}.

\For{$m= M_0,\ldots,M+1$}
{
Update the distance matrix by (\ref{equ:dist}).\\
Find two clusters with the minimal  distance, and  merge them into one single cluster.
 }

 Calculate the cluster centers
 and weights via \eqref{equ:clutercenter}.

\Output{Clustered nodes and weights $ \{ \vecpar_k,w_k\}_{k=1}^M$.}
\end{algorithm}

 Clustering is a classical technique in pattern recognition and data mining~\cite{jain1999data}, and it gathers data with similar pattern into one group. A widely used algorithm is hierarchical clustering. At the beginning, each single data point is a cluster by its own, then two clusters with ``the minimal distance'' are merged into one single cluster sequentially. Consequently, the number of clusters is decreased by one in each iteration until the prescribed number of clusters is reached.
The widely used hierarchical approaches includes single linkage, complete linkage and average linkage.
They mainly differ in the criterion of choosing ``the  distance''. The complete-linkage clustering chooses the distance between two clusters $C_i$ and $C_j$ as
\begin{equation*}
D^0_{ij}=\max_{\vecpar_1\in C_i, \vecpar_2 \in C_j} d(\vecpar_1,\vecpar_2),
\end{equation*}
where $d(\vecpar_1, \vecpar_2)=\|\vecpar_1- \vecpar_2\|_2$.
In our problem, the sample points are equipped with some weight parameters, therefore, we modify the complete-linkage clustering and consider a weighted clustering problem.

 \textbf{Weighted Complete Linkage Clustering.} We define the weighted distance as
\begin{equation}\label{equ:dist}
D_{ij}=(w_i+w_j)\left(\max_{\vecpar_1\in C_i, \vecpar_2 \in C_j} d(\vecpar_1,\vecpar_j)\right),
\end{equation}
where $w_i=\sum_{\vecpar_k\in C_i} w(\vecpar_k)$
is the weight of the $i$-th cluster.
The above distance considers both the geometric distance and the weights of different clusters. The intuition behind \eqref{equ:dist} is that we do not want a sample with a very small weight to form a cluster by itself. This algorithm tends to group a sample with a very small weight with its nearest cluster.

Once the number of clusters reduces to $M$, we stop the iterations and return the weight and cluster center as
\begin{equation}\label{equ:clutercenter}
w_i=\sum_{\vecpar_k\in\mat{C}_i}w(\vecpar_k), \;  \vecpar_i=\sum_{\vecpar_k\in\mat{C}_i} \frac{w(\vecpar_k)}{w_i}\vecpar_k,\; \forall\, i=1,\ldots,M.
\end{equation}

Algorithm~\ref{alg:cluster} has summarized the pseudo codes of our clustering method used to initialize Algorithm~\ref{alg:update_coeff}.

% choosing the number of points
\subsection{Number of Quadrature Points}
A fundamental question is: how many quadrature samples are necessary in order to achieve a desired level of accuracy? This question is well answered in the one-dimensional Gauss quadrature rule: $p$ quadrature points provide an exact result for the numerical integration of any polynomial function bounded by order $2p-1$~\cite{Golub:1969}. However, there is no similar result for general multidimensional correlated cases.

Let ${\cal S}_{2p}$ denote all polynomial functions of $\vecpar$ with their total orders bounded by $2p$. The integration rule $\{ \vecpar_k, w_k\}_{k=1}^M$ has a $2p$-th-order accuracy if \eqref{equ:nmint1} is satisfied. Here the $2p$-th-order accuracy means that $\sum \limits_{k=1}^M f(\vecpar_k)w_k =\mathbb{E} [ f(\vecpar)]$ for {\it any} $f(\vecpar) \in {\cal S}_{2p}$. We have the following result on the number of quadrature samples in order to ensure the $2p$-th-order accuracy.

\begin{theorem}\label{thm:Mbound}
Assume that $M$ pairs of quadrature samples and weights are obtained from \eqref{equ:nmint1} to ensure the $2p$-th-order integration accuracy, then the number of quadrature points satisfies $N_p \leq M \leq N_{2p}$.
\end{theorem}

\begin{proof}
See Appendix \ref{append:Mbound} for the details.
\end{proof}

\ccf{ While there exists at least one $M$ in $[N_p,N_{2p}]$ such that the $2p$-th-order integration accuracy can be achieved, we can have multiple choices of $M$, and we may even have multiple choices of quadrature samples and weights for each $M$. In our stochastic collocation framework, we only require one (among possibly multiple) set of quadrature samples and weights with a sufficiently small $M$. }

In practice, we try to get a better solution by generating a better initial guess. We do this by firstly generate $6N_p$ random samples via Monte Carlo, and group them into $2N_p$ clusters. These $M=2N_p$ samples are used as the initial quadrature points. Then, we increase or decrease $M$ via Algorithm~\ref{alg:update_coeff}. This process is illustrated via a 2-dimensional example in Fig~\ref{fig_quadrature_synth}. The practical number of quadrature nodes used by our stochastic collocation framework is very close to the theoretical lower bound, which is experimentally shown in Section~\ref{sec:result_num}.

%  \begin{algorithm}[t]
% \label{alg:updatenodes}
% \caption{Calculating the  quadrature nodes with automatic sampling number determination }
%  %\SetAlgoNoLine
%       \SetKwInput{Input}{Input}
%       \SetKwInput{Output}{Output}
% \Input{Initial  point  number $M_0$, the maximal point number $M_{\max}$ and minimal point number $M_{\min}$.}
% Generate the initial nodes via Algorithm \ref{alg:cluster}. \\
% Calculate $\mat{W}$ and $\vecpar$  by solving~\eqref{equ:NLS}. \\
% \While{The optimization does not converge}
% {Let $M=M+1$, regenerate the initial nodes via Algorithm \ref{alg:cluster} and  execute Algorithm \ref{alg:nls}.
% }% while

% \While{Algorithm \ref{alg:nls} converges}
% {Let $M=M-1$, delete the point  with the smallest weight and  execute Algorithm \ref{alg:nls}. }% while
% \Output{Optimal quadrature nodes $\vecpar_1,\ldots,\vecpar_{M}$ and weights $\mat{w}$}
% \end{algorithm}

\section{Theoretical Error Bounds}
\label{sec:theory}

In this section, we provide several theoretical results regarding the numerical accuracy of our proposed stochastic collocation algorithm for non-Gaussian correlated cases.

\subsection{Conditions for Exact Results}
The following theorem show that our quadrature rule \eqref{equ:nmint1} can provide exact results if $y(\vecpar)$ satisfies certain conditions.

%% lemma
\begin{theorem}
Suppose that $y(\vecpar)\in\mathcal{S}_p$ is a polynomial function bounded by order $p$, i.e., there exist some coefficients $\{c_{\basisInd}\}$ such that $
y(\vecpar)=\sum_{|\basisInd|=0}^p c_{\basisInd} \multiGPC_{\basisInd}(\vecpar)$.
Denote the approximated expansion obtained via our numerical integration as
\begin{equation}\label{equ:yapp}
\tilde{y}(\vecpar)=\sum_{|\basisInd|=0}^p \tilde c_{\basisInd} \multiGPC_{\basisInd}(\vecpar), \text{ with } \tilde c_{\basisInd}=\sum_{k=1}^M y(\vecpar_k)\multiGPC_{\basisInd} (\vecpar_k)w_k.
\end{equation}
Then $y(\vecpar)$ can be recovered exactly, i.e., $y(\vecpar)=\tilde y(\vecpar)$, if $\{\vecpar_k, w_k\}$ satisfies \eqref{equ:nmint1} strictly for all $j=1,\ldots,N_{2p}$.
\label{thm:exactrecovery}
\end{theorem}

\begin{proof}
The detailed proof is provided in Appendix~\ref{app:lem1}.
\end{proof}

In practice, we may not be able to get an exact solution because of two reasons: (1) $y(\vecpar)$ is not a polynomial in ${\cal S}_p$; (2) the quadrature points and weights obtained by our numerical nonlinear optimization solver causes a small residual in \eqref{equ:nmint1}. In this case, we can provide an error bound for our solution when $y(\vecpar)$ is smooth enough and when the nonlinear optimization problem \eqref{equ:NLS} is solved with certain accuracy (i.e., when the resulting objective function is below a threshold).

%   Even for $y(\vecpar)\notin \ten{S}_p$, our approximation $\tilde{y}(\vecpar)$ can fully capture the structure in  $\ten{S}_p$£¬
% and the   error only exits in the orthogonal complement space of  $\ten{S}_p$.
% \begin{lemma}\label{equ:truncerror}
% For any function $y(\vecpar)$, denote  $\tilde{y}(\vecpar)$ as its approximation function in $\ten{S}_p$ define as \eqref{equ:yapp}. Then   the truncation error  only exits in the orthogonal complement space of $\ten{S}_p$.
% That is to say, for any $|\basisInd|\le p$, there is
% \begin{equation}
% \mathbb{E}[(y(\vecpar)-\tilde{y}(\vecpar))\multiGPC_{\basisInd}(\vecpar)]=0.
% \end{equation}
% \end{lemma}

% The above lemma can be shown
% directly by the fact that $\mathbb{E}[y(\vecpar)\multiGPC_{\basisInd}(\vecpar)]=\mathbb{E}[y_p(\vecpar)\multiGPC_{\basisInd}(\vecpar)]=c_{\basisInd}$.

\subsection{Three Weak Assumptions}

In order to provide a theoretical analysis for the numerical error caused by $y(\vecpar)$ and by the nonlinear optimization solver, we make the following weak assumptions.

\textbf{Assumption 1.} $y(\vecpar)$ is squared integrable. In other words, there exists a positive scalar $L$ such that
\begin{equation}\label{equ:bdy}
\|y(\vecpar)\|_2=\sqrt{\mathbb{E}[y^2(\vecpar)]}\le L.
\end{equation}
Denote $y_p(\vecpar)=\arg\min_{\hat{y}(\vecpar)\in\ten{S}_p}\|y(\vecpar)-\hat y(\vecpar)\|_2$ as the projection of $y(\vecpar)$ onto $\ten{S}_p$.
We assume that  there exists $\delta\ge0$ such that
\begin{equation}\label{equ:apperr}
\|y(\vecpar) - y_p(\vecpar)\|_2\le \delta.
\end{equation}
Actually $y_p(\vecpar)$ can be written as $y_p(\vecpar)=\sum_{|\basisInd|=0}^p c_{\basisInd}\multiGPC_{\basisInd}(\vecpar)$, where $c_{\basisInd}=\mathbb{E}[y(\vecpar)\multiGPC_{\basisInd}(\vecpar)]$.

\textbf{Assumption 2.} Define the numerical integration operator
\begin{equation}
\mathbb{I}[y(\vecpar)]=\sum_{k=1}^M y(\vecpar_k)w_k.
\end{equation}
We assume that the operator $\mathbb{I}[y(\vecpar)]$ is bounded, i.e., there exists $W\ge0$ such that
\begin{equation}\label{equ:Ibd}
|\mathbb{I}[y(\vecpar)] |\le W \|y(\vecpar)\|_1, \; {\rm where} \; \|y(\vecpar)\|_1=\mathbb{E}[|y(\vecpar)|].
\end{equation}

\textbf{Assumption 3.} The nonlinear least square problem \eqref{equ:NLS} is solved with an error threshold $\epsilon\ge 0$, i.e.,
\begin{equation}\label{equ:numerr}
 \|\mat{\Phi}(\bar{\vecpar})\mat{w}-\mat{e_1}\|_1\le \epsilon,
\end{equation}
where $\|\cdot\|_1$ denotes the $\ell_1$ norm in the Euclidean space.
Here the $j$-th element in the vector $\mat{\Phi}(\bar{\vecpar})\mat{w}-\mat{e_1}$ actually can be written as
$\mathbb{I}[\multiGPC_j(\vecpar)]-\mathbb{E}[\multiGPC_j(\vecpar)]$.

 \subsection{Error Bound of the Proposed Stochastic Collocation}

\begin{theorem} \label{thm:interr}
 Suppose that Assumptions 1-3 hold, then numerical integration error satisfies
\begin{equation}
\left|\mathbb{E}[y(\vecpar)]-\mathbb{I}[y(\vecpar)]\right|\le L\epsilon+W\delta.
\end{equation}
Here $L$ is the upper bound of $\|y(\vecpar)\|_2$ in \eqref{equ:bdy}, $W$ is the upper bound of the numerical integration $\mathbb{I}[y(\vecpar)]$ in \eqref{equ:Ibd}, $\epsilon$ is the numerical error of our nonlinear optimization solver defined in \eqref{equ:numerr}, and $\delta$ is the distance from $y(\vecpar)$ to $\ten{S}_p$ in \eqref{equ:apperr}.
\end{theorem}
\begin{proof}
See Appendix \ref{app:thm1}.
\end{proof}

Based on Theorem \ref{thm:interr}, we can further derive an upper bound for the following approximation error.

\begin{theorem}\label{thm:apprerr}
With Assumptions 1-3, the numerical error of our stochastic collocation algorithm satisfies
% \begin{equation}
% \|y(\vecpar)-\tilde y(\vecpar)\| \le N_p^d D^2 \delta+(N_p^d)^2N_{2p}^d D \delta,
% \end{equation}
\begin{equation}\label{equ:sc_err}
\|y(\vecpar)-\tilde y(\vecpar)\|_2\le \delta + N_p(LT\epsilon+W\delta),
\end{equation}
where $T=\max_{j,l=1,\ldots,N_{2p}}\|\multiGPC_j(\vecpar)\multiGPC_l(\vecpar)\|_2$.
\end{theorem}
\begin{proof}
See Appendix \ref{app:thm2} for the details.
\end{proof}

{\bf Remarks:} Theorem~\ref{thm:apprerr} indicates the following intuitions:
\begin{enumerate}
\item if the nonlinear optimization solver is accurate enough and $\epsilon$ is very small, the error of our stochastic collocation is dominated by the approximation error $\delta$;
\item as we increase the order of basis functions, $\delta$ decreases and the result becomes more and more accurate;
\item if the total order of the basis function is very high and $\delta$ becomes extremely small, the optimization error $\epsilon$ will dominate the overall numerical error, and the convergence will slow down.
\end{enumerate}

Once \eqref{equ:nmint1} holds, we should have the following result
\begin{equation*}
\mathbb{I}[\multiGPC_i(\vecpar)\multiGPC_j(\vecpar)] =\sum_{k=1}^M  \multiGPC_{i}(\vecpar_k)\multiGPC_{j}(\vecpar_k)w_k=\delta_{i,j}.
\end{equation*}
In practice, there are numerical errors caused by quadrature points and weights obtained by the optimization solver. In the following lemma, we show that the error is bounded.

\begin{lemma}\label{lemma:coverr}
Suppose that Assumptions 1-3 hold, define a matrix $\mat{V}\in\mathbb{R}^{N_p\times N_p}$ with each element $\mat{V}_{ij}=\mathbb{I}[\multiGPC_i(\vecpar)\multiGPC_j(\vecpar)]$ being a numerical evaluation of
$\mathbb{E}[\multiGPC_i(\vecpar)\multiGPC_j(\vecpar)]$ using the quadrature points and weights from solving \eqref{equ:NLS}.
We have
\begin{equation}\label{equ:errcov}
\|\mat{V}-I_{N_p}\|_F\le N_pT \epsilon.
\end{equation}
\end{lemma}

\begin{proof}
See Appendix \ref{app:coverr}.
\end{proof}
% % error of
% \section{Implementation details}
% %% initial quadrature nodes

% \subsection{Accuracy analysis}
% Suppose that we want to approximate $\out(\vecpar)$ by a $p$-th order polynomial functions, then $2p$-th order numerical integration will be used in (\ref{eq:yproject}). Consequently, (\ref{equ:NLSw}) involved $2p$-th order basis matching.  In order to calculate the $2p$-th order basis functions, $4p$-th moment computations are involved in
% (\ref{equ:project}).

% %table of complexity
% \begin{table}
% \caption{Table of complexity}
% \begin{tabular}{ccc}
% \hline
% & order & complexity \\
% expansion & $p$ & \\
% integral & $2p$ & \\
% moments & $4p$ & \\
% \hline
% \end{tabular}
% \end{table}

\begin{figure*}[t]
	\centering
		\includegraphics[width=5.8in]{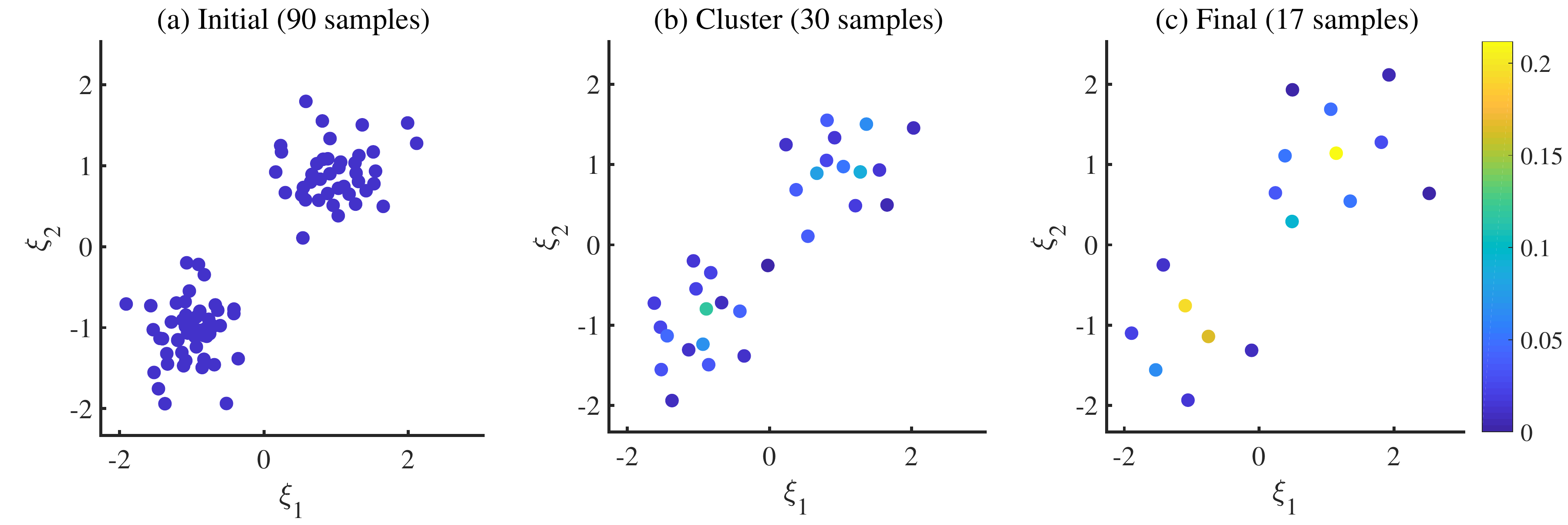}
  \caption{The process of generating quadrature samples and weights for the synthetic example. The quadrature weights are shown by the color bar. (a) Initial candidate points generated via Monte Carlo; (b) clustered samples via the weighted complete linkage method in Algorithm~\ref{alg:cluster}; (c) the optimized quadrature nodes by our Algorithm \ref{alg:update_coeff}. This process only depends on the probability density function and the basis functions, and is independent of $y(\vecpar)$.}
  \label{fig_quadrature_synth}
 \end{figure*}

\begin{figure}[t]
	\centering
\includegraphics[width=75mm]{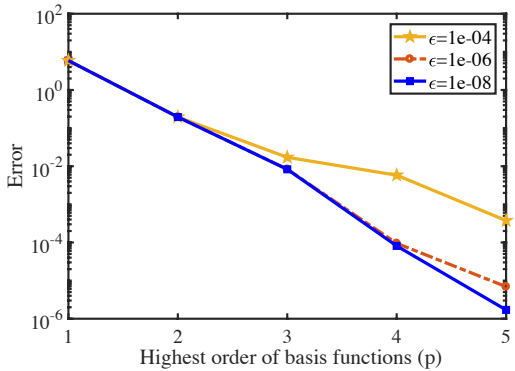}
\caption{Convergence rate for the synthetic example. Here $\epsilon$ is the numerical error of optimization defined in \eqref{equ:numerr}. This figure demonstrates the error estimated in \eqref{equ:sc_err}: the stochastic collocation algorithm shows a nearly exponential convergence rate as $p$ increases and before $\epsilon$ dominates the error.}
	\label{fig:synth}
\end{figure}

\section{Numerical Results}
\label{sec:numerical}
In order to show the efficiency of our proposed method, we conduct numerical experiments on a synthetic example, a 3-stage CMOS electronic ring oscillator, and an optical filter. The stopping criterion in \eqref{equ:NLS} is set as $\epsilon=10^{-8}$ unless stated otherwise. In all examples, we use some Gaussian mixture models to describe the joint density functions of correlated non-Gaussian random parameters. \ccf{The Matlab codes and a demo example are provided online at: \url{https://web.ece.ucsb.edu/~zhengzhang/codes_dataFiles/uq_ng}. }

%%%% Synthetic Example
\subsection{A Synthetic Example}

Firstly we consider a synthetic example, and use it to show the accuracy and convergence rate of our proposed stochastic collocation algorithm. Specifically, we consider the following smooth function of two correlated parameters
\begin{equation}
\label{ex:synth}
\out (\vecpar) =  \exp(\xi_1)  +0.1\cos(\xi_1)\sin(\xi_2).
\end{equation}
We assume that the random parameters follow a Gaussian mixture distribution
\begin{equation}\label{equ:synth}
\vecpar = \vecpar_0+\frac{1}{10}\Delta\vecpar,\text{ where } \Delta\vecpar \sim \frac12 \mathcal{N}(\boldsymbol \mu_1, \boldsymbol \Sigma_1) + \frac12\mathcal{N}(\boldsymbol \mu_2, \boldsymbol \Sigma_2). \nonumber
\end{equation}
Here, the mean values $\boldsymbol \mu_1=\mat{1}$,  $\boldsymbol \mu_2=-\mat{1}$; the positive definite covariance matrices
$\boldsymbol \Sigma_1$ and $\boldsymbol \Sigma_2$ are randomly generated. We use $\mat{1}$ to denote a vector of a compatible size with all elements being one. We will also use this notation in other examples.

We first illustrate how to generate the quadrature samples and weights by our optimization-based quadrature rule. Assume that we want to approximate $y(\vecpar)$ by a forth-order expansion. Firstly, 90 random samples are generated via Monte Carlo. Secondly, these points are grouped into 30 clusters via our proposed weighted linkage clustering approach, and they are used as the initial samples and weights of Algorithm~\ref{alg:update_coeff}. Finally, the number of quadrature nodes is reduced to 17 automatically by Algorithm~\ref{alg:update_coeff}, whereas the lower bound for the number of quadrature nodes is 15. The process of generating quadrature samples and weights is shown in Fig.~\ref{fig_quadrature_synth}.

\begin{table}
\centering
\caption{Accuracy comparison on the synthetic experiments. The underscores indicate precision.}
\begin{tabular}{c|c|c|c|c|c}
\hline
method& \multicolumn{5}{c}{Proposed}  \\ \hline
$p$ & 1&2&3&4&5\\\hline
\# samples & 3&6&10&17&66\\ \hline
mean & 2.7\underline{8}35& 2.782\underline{9}& 2.782\underline{9}& 2.782\underline{9}&  2.782\underline{9}\\
\thickhline
method&\multicolumn{5}{c}{Monte Carlo} \\ \hline
\# samples &$10$ & $10^2$& $10^3$& $10^4$&$10^5$\\ \hline
mean &  \underline{2}.6799 & 2.\underline{7}625 &2.\underline{7}911&2.7\underline{8}11 &2.782\underline{9}\\
% mean & 1.\underline{9}2303& 1.\underline{9}3840&1.9\underline{4}307& 1.94\underline{1}32& 1.94183\\
\hline
\end{tabular}
\label{tab:synthetic}
\end{table}

Theorem \ref{thm:apprerr} shows that the error depends on two parts: the numerical error $\epsilon$ of the optimization solver of our quadrature rule, and the approximation error $\delta$ by order-$p$ basis functions.
When $p$ is small, $\|y(\vecpar)-y_p(\vecpar)\|_2 \leq \delta$ dominates the error.
When $p$ is large, $\delta$ becomes small and $\epsilon$  dominates the error, therefore smaller $\epsilon$ will produce more accurate results.
In order to verify this theoretical result, we perform stochastic collocation by using different orders of basis functions (i.e., $p=1$ to $5$) and by setting different error thresholds (i.e., $\epsilon=10^{-4}$, $ 10^{-6}$ and $ 10^{-8}$) in the optimization-based quadrature rule. As shown in Fig.~\ref{fig:synth}, our stochastic collocation has a nearly exponential convergence rate before $\epsilon$ dominates the error.

We further compare our our method with Monte Carlo in Table~\ref{tab:synthetic}.
Our method provides a closed-form expression for the mean value of $y(\vecpar)$, and a 2nd-order expansion using $6$ quadrature points is sufficient to achieve a precision of $4$ fractional digits. In contrast, Monte Carlo requires $10^5$ random samples to achieve the similar level of accuracy.

\begin{figure}[t]
	\centering \includegraphics[width=3.2in]{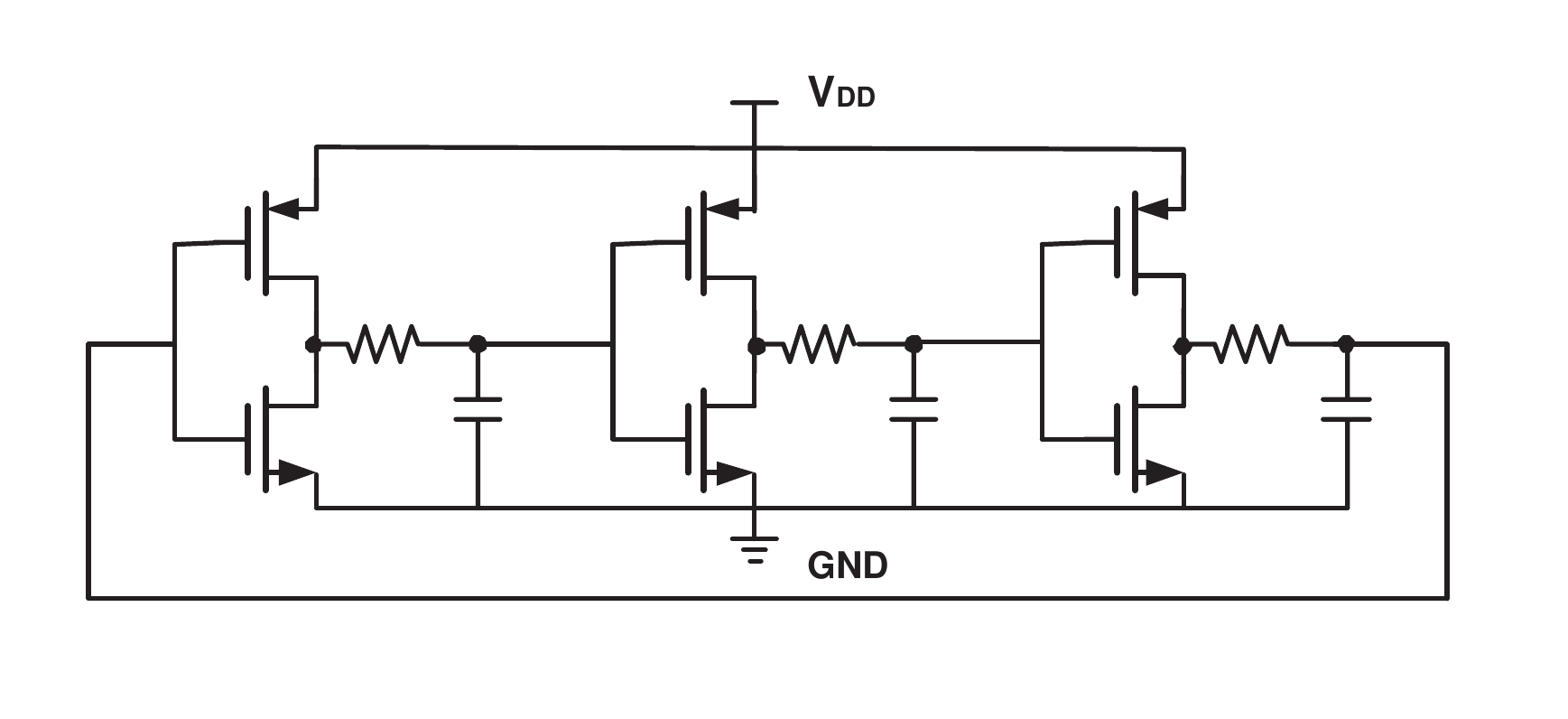}
\caption{Schematic of a 3-stage CMOS ring oscillator. }
	\label{fig:ring}
\end{figure}

\begin{figure*}[t]
\centering
\includegraphics[width=5.5in] {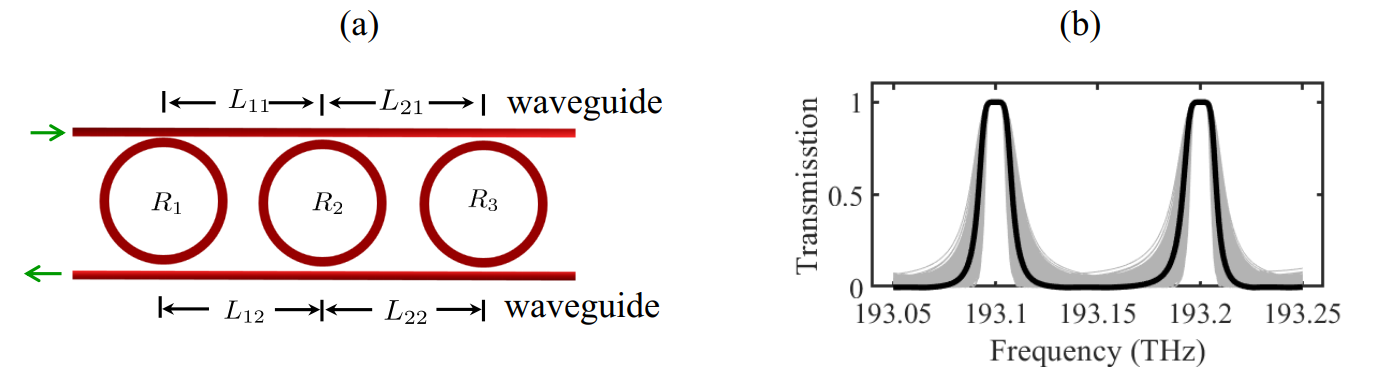} % height=1.7in
\caption{(a) Schematic of a 3-stage parallel-coupled ring resonator optical filter. $L_{12}$, $L_{21}$, $L_{23}$ and $L_{32}$ are the connecting waveguides, and $R_1$, $R_2$ and $R_3$ denote the rings.
(b) The black line shows the nominal transmission function, and the thin grey lines show the effect of fabrication uncertainties on the waveguide lengths of $L_{12}$, $L_{21}$, $L_{23}$, $L_{32}$.
}
\label{fig:coupled_ring_resonator}
\end{figure*}

\subsection{A 3-Stage CMOS Electronic Ring Oscillator}

We continue to verify our algorithm by the 3-stage CMOS ring oscillator in Fig.~\ref{fig:ring}. We model the relative threshold voltage variations of six transistors via
 \begin{equation}
\vecpar = \vecpar_0+\mat{D} \Delta\vecpar,\; \text{with }\Delta\vecpar\sim  \frac23 \mathcal{N}(\boldsymbol \mu_1, \boldsymbol \Sigma_1) + \frac13\mathcal{N}(\boldsymbol \mu_2, \boldsymbol \Sigma_2), \nonumber
\end{equation}
where $\mat{D}$ is a diagonal scaling matrix, $\mat{\mu}_1=\mat{1}$, $\mat{\mu}_2=-\mat{1}$, and
$\boldsymbol \Sigma_1$ and $\boldsymbol \Sigma_2$  are randomly generated positive definite matrices.

 We aim to approximate the frequency by a 2nd-order expansion of our multivariate basis functions. Our optimization-based quadrature rule generates 33 pairs of quadrature samples and weights, then a deterministic periodic steady-state simulator is called repeatedly to simulate the oscillator at all parameter samples. Fig.~\ref{fig:ring_results}
shows the obtained weights of all basis functions and the probability density function. %Our approach can achieve the the similar level of accuracy compared with Monte Carlo using $10^5$ random simulations.

We compare the computed mean value  from our methods with that from Monte Carlo in Table~\ref{tab:ring}. Monte Carlo method converges very slowly, and requires $3030\times$ more simulation samples to achieve the similar level of accuracy (with 2 accurate fractional digits).

\begin{figure}[t]
	\centering
\includegraphics[width=90mm, height=1.5in]{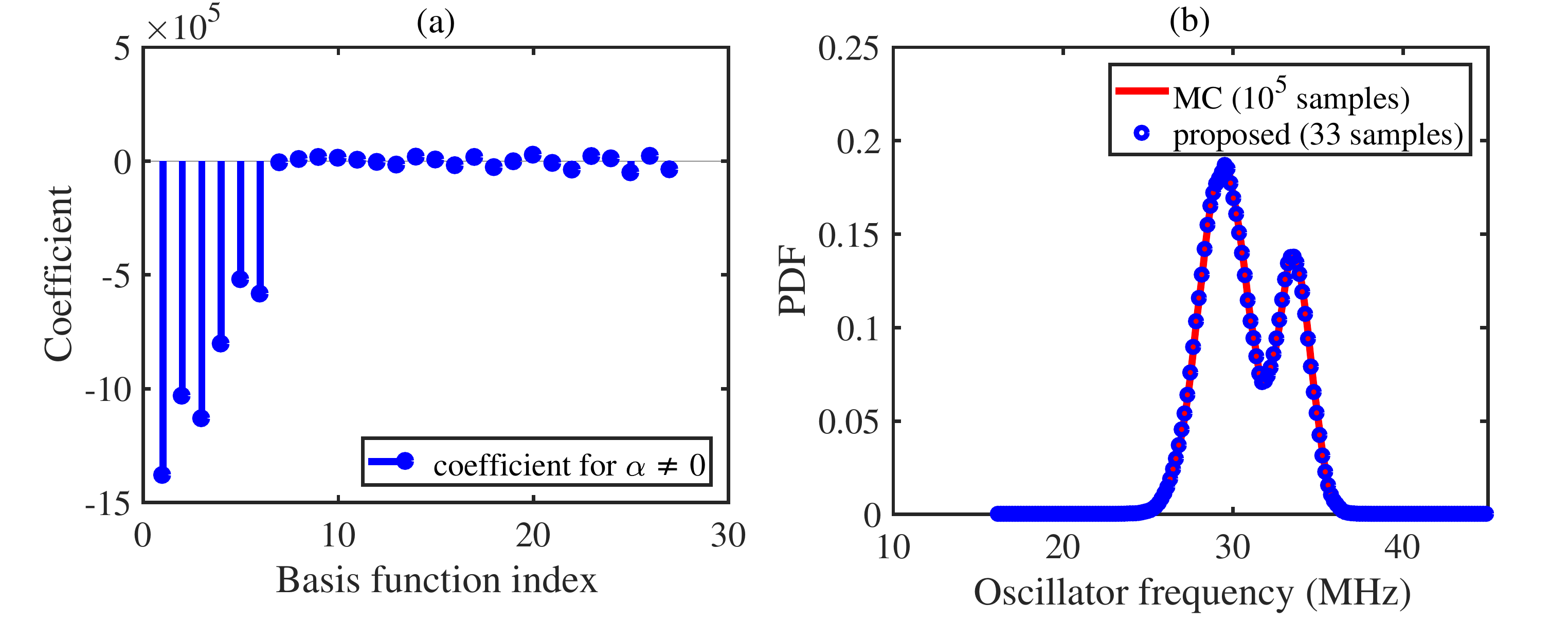}
\caption{Numerical results of the CMOS ring oscillator. (a) obtained coefficients/weights of our basis functions; (b) probability density functions of the oscillator frequency obtained by our proposed method and Monte Carlo (MC).}
	\label{fig:ring_results}
\end{figure}

\begin{table}[t]
\centering
\caption{Accuracy comparison on the CMOS ring oscillator. The underscores indicate precision.}
\begin{tabular}{c|c|c|c|c|c}
\hline
method& Proposed &\multicolumn{4}{c}{Monte Carlo} \\ \thickhline
\# samples & 33 & $10^2$& $10^3$& $10^4$&$10^5$\\ \hline
mean (MHz) &30.8\underline{3}&3\underline{0}.93&30.\underline{8}8& 30.\underline{8}0&30.8\underline{3} \\
\hline
\end{tabular}
\label{tab:ring}
\end{table}

\subsection{A Parallel Coupled Ring Resonator Optical Filter}

In this subsection, we consider the 3-stage parallel-coupled ring resonator optical filter\footnote{The details of this benchmark can be found at \url{https://kb.lumerical.com/en/pic_circuits_coupled_ring_resonator_filters.html}} in Fig.~\ref{fig:coupled_ring_resonator} (a). This optical filter is a versatile component for wavelength filtering, multiplexing, switching, and modulation in photonic integrated circuits. This circuit has a nominal 3-dB bandwidth of 12 GHz, and the coupling coefficients for the three rings are $K_1=K_3=0.198836$ and $K_2=0.356423$. In the nominal design, the waveguide lengths $L_{12},L_{21},L_{23},L_{32}$ are all $30.6624$ $ \mu$m, and the circumference of all ring are $R_{1}=R_2=R_3=2997.92 $ $\mu$m. In practice, there exist non-Gaussian correlated uncertainties in the waveguide geometric parameters. The effect of fabrication uncertainties are shown in
Fig.~\ref{fig:coupled_ring_resonator} (b).

 \begin{figure}[t]
	\centering  \includegraphics[width=90mm,height=1.6in]{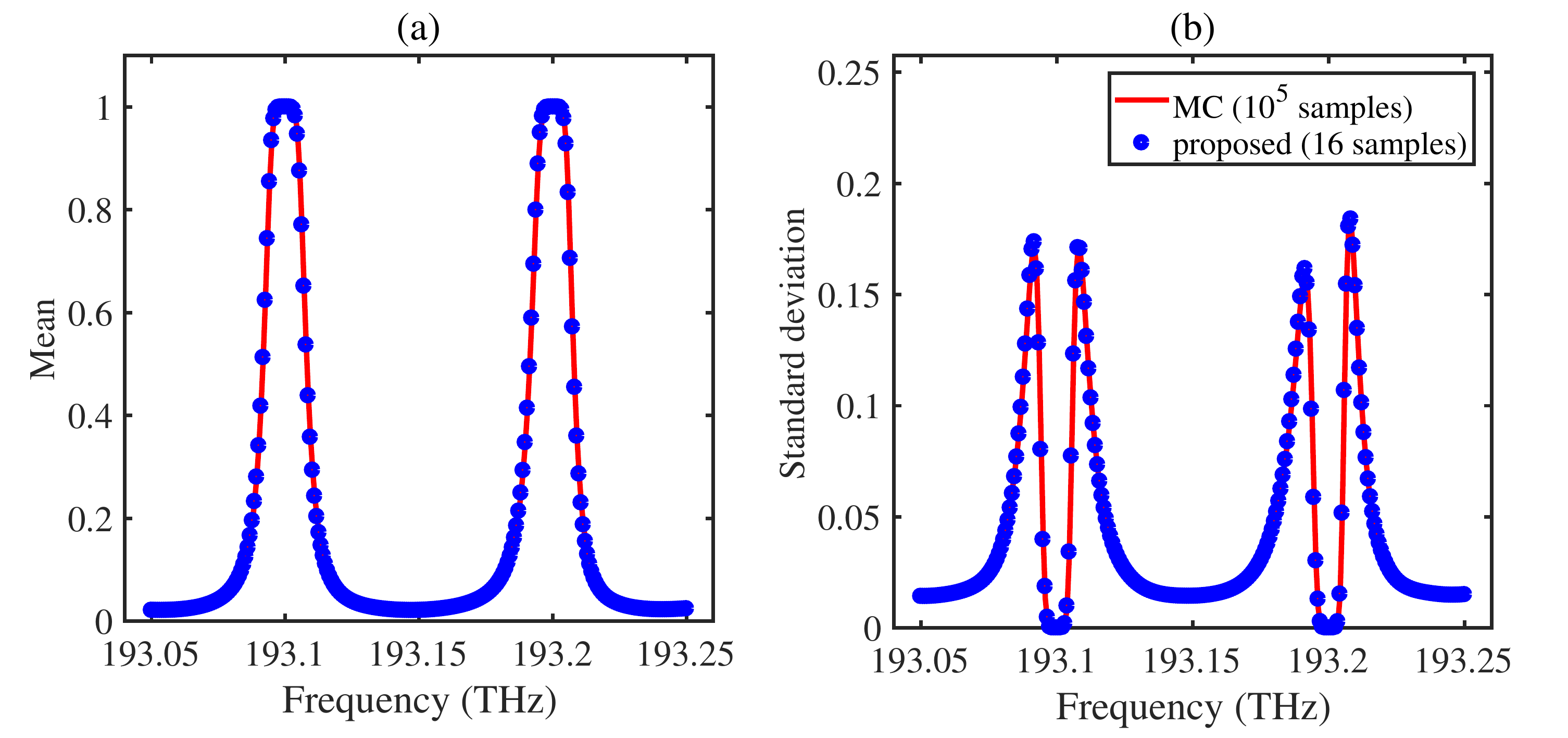}
\caption{Simulation results  with respect to the geometric uncertainties in the waveguide length of $L_{12}$, $L_{21}$, $L_{23}$ and $L_{32}$.  (a) obtained mean value of the power transmission rate; (b) standard deviation of the transmission rate.} \label{fig:res_ringresonator_4000}
\end{figure}

 \begin{figure}[t]
	\centering  \includegraphics[width=90mm,height=1.6in]{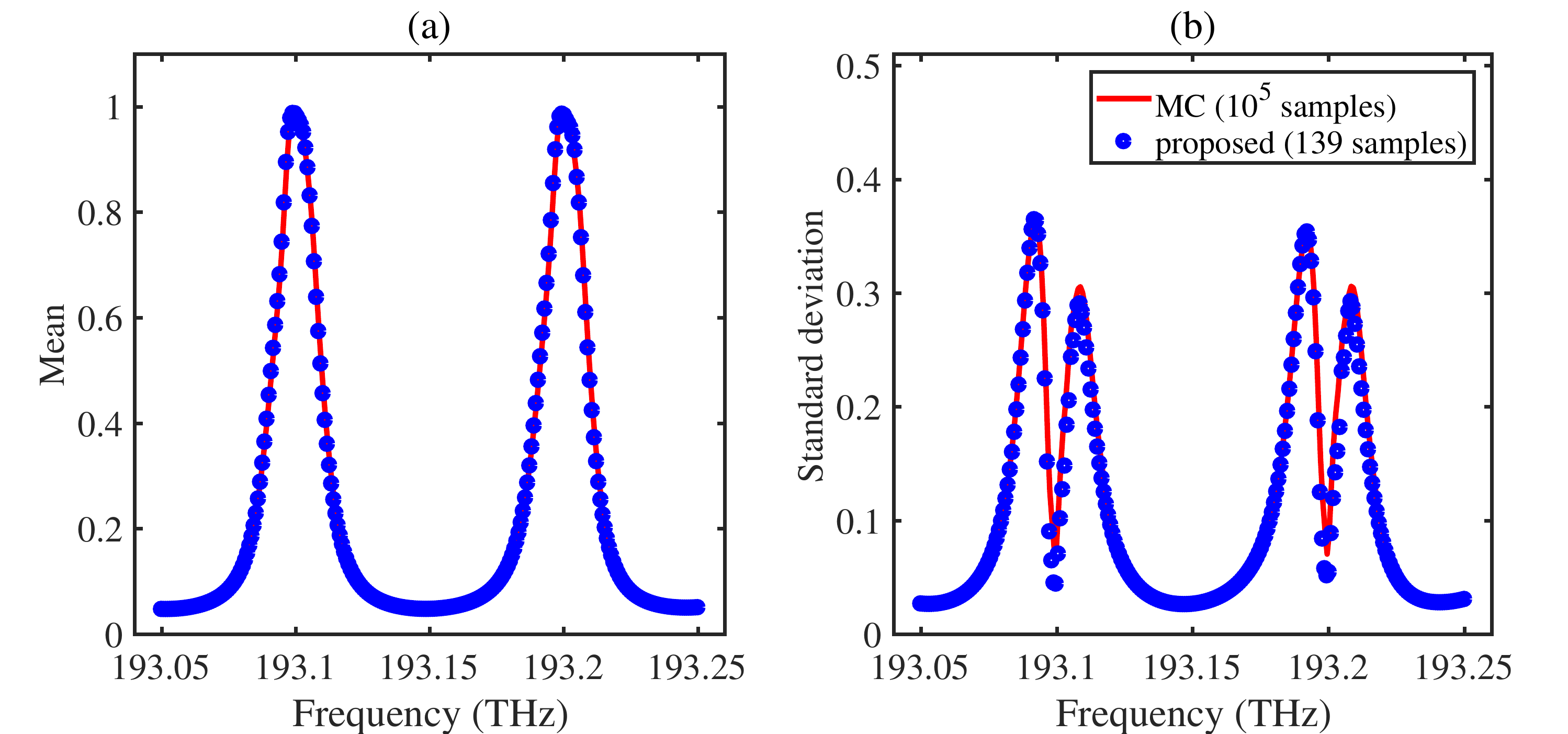}
\caption{Simulation results with respect to the geometric uncertainties in the waveguide length of  $L_{12}$, $L_{21}$, $L_{23}$, $L_{32}$, $R_1$, $R_2$, $R_3$, and the uncertainties in effective index for  $L_{12}$, $L_{21}$, $L_{23}$ and $L_{32}$.  (a) obtained mean value of the power transmission rate; (b) standard deviation of the transmission rate.} \label{fig:res_ringresonator_4430}
\end{figure}

Our goal is to build a $2$nd-order stochastic model to approximate the power transmission curve at different frequency points $\out(f,\vecpar)=\sum_{|\basisInd|=0}^p c_{\basisInd}(f) \multiGPC_{\basisInd}(\vecpar)$.
We use a Gaussian mixture model to describe the uncertainties,
\begin{equation}
\vecpar = \vecpar_0+\Delta\vecpar,\text{ where }\Delta\vecpar\sim \frac12\mathcal{N}(\boldsymbol{\mu}_1,\boldsymbol{\Sigma}_1)+\frac12\mathcal{N}(\boldsymbol{\mu}_2,\boldsymbol{\Sigma}_2).\nonumber
\end{equation}
For the  waveguide length parameters, we use
\begin{equation*}
\boldsymbol{\mu}_1= -\boldsymbol{\mu}_2=25 \times \mat{1} {\rm nm},\   \boldsymbol{\Sigma}_1=\boldsymbol{\Sigma}_2=6.25(\mathbf{I}+0.5\mathbf{E}).
\end{equation*}
The uncertainties of the effective index follows a Gaussian mixture distribution with
\begin{equation*}
\boldsymbol{\mu}_1= -\boldsymbol{\mu}_2= 10^{-3}\times \mat{1},  \ \boldsymbol{\Sigma}_1=\boldsymbol{\Sigma}_2=10^{-6}(\mathbf{I}+0.5\mathbf{E}).
\end{equation*}

We perform two experiments for the optical filter. The first experiment only considers the uncertainties of the waveguide lengths $L_{12}$, $L_{21}$, $L_{23}$ and $L_{32}$.
The second experiments considers uncertainties in the waveguide lengths $L_{12}$, $L_{21}$, $L_{23}$, ring geometry $L_{32}$, $R_1$, $R_2$ and $R_3$, as well as the effective index in $L_{12}$, $L_{21}$, $L_{23}$ and $L_{32}$.
The mean value and standard derivation of the output response are shown in Fig.~\ref{fig:res_ringresonator_4000} and Fig.~\ref{fig:res_ringresonator_4430}, respectively.
Although our method only uses 16 or 139 samples, it is able to achieve the similar accuracy with Monte Carlo that consumes $10^5$ simulation samples.

\begin{figure*}[t]
	\centering \includegraphics[width=5.5in]{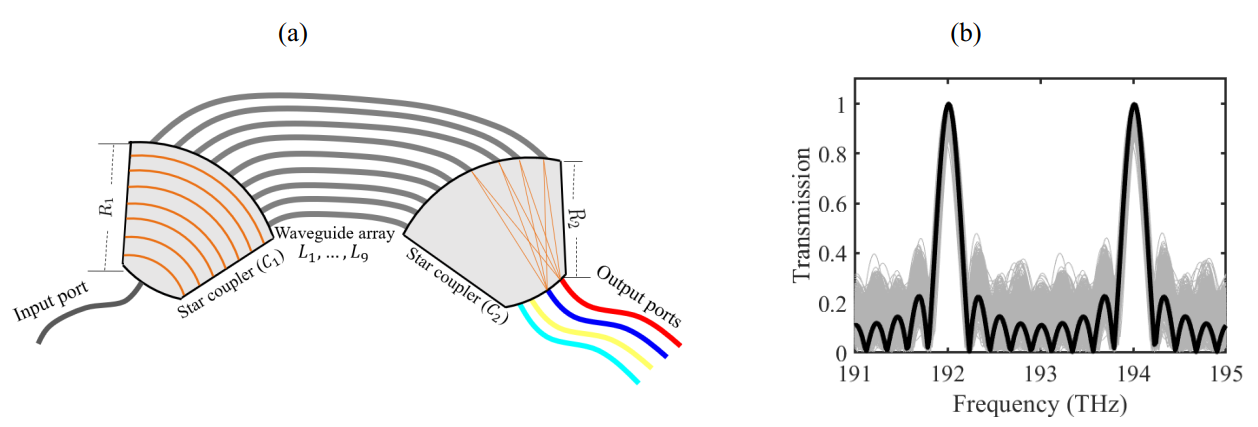}
\caption{(a) Schematic of an AWG with 9 waveguide arrays; (b) The nominal transmission rate from the input to output Port 1. The black curve shows the result without any uncertainties, and the grey lines show the effects caused by the fabrication uncertainties of radius $R_1$, $R_2$ and waveguide lengths $L_1,\ldots,L_9$.}
	\label{fig:AWG}
\end{figure*}

\subsection{An Arrayed Waveguide Grating (AWG)}
Finally, we consider an arrayed waveguide grating (AWG)~\cite{zhang2018verilog}. The AWG is essential for wavelength division and multiplexing in photonic systems. In our experiment, we use an AWGR with 9 waveguide arrays and two star couplers, as shown in Fig. \ref{fig:AWG} (a). In the nominal design, the radius of each star coupler is $R_1=R_2= 2.985 $ mm, and the waveguide lengths $L_1,\ldots,L_9$ range from $46$ $ \mu $m to $420$ $ \mu$m. In practice, there exist non-Gaussian correlated uncertainties in the device geometric parameters, and the resulting performance uncertainties are shown in Fig.~\ref{fig:AWG}~(b).

\begin{figure}[t]
	\centering
	\includegraphics[width=90mm,height=1.8in]{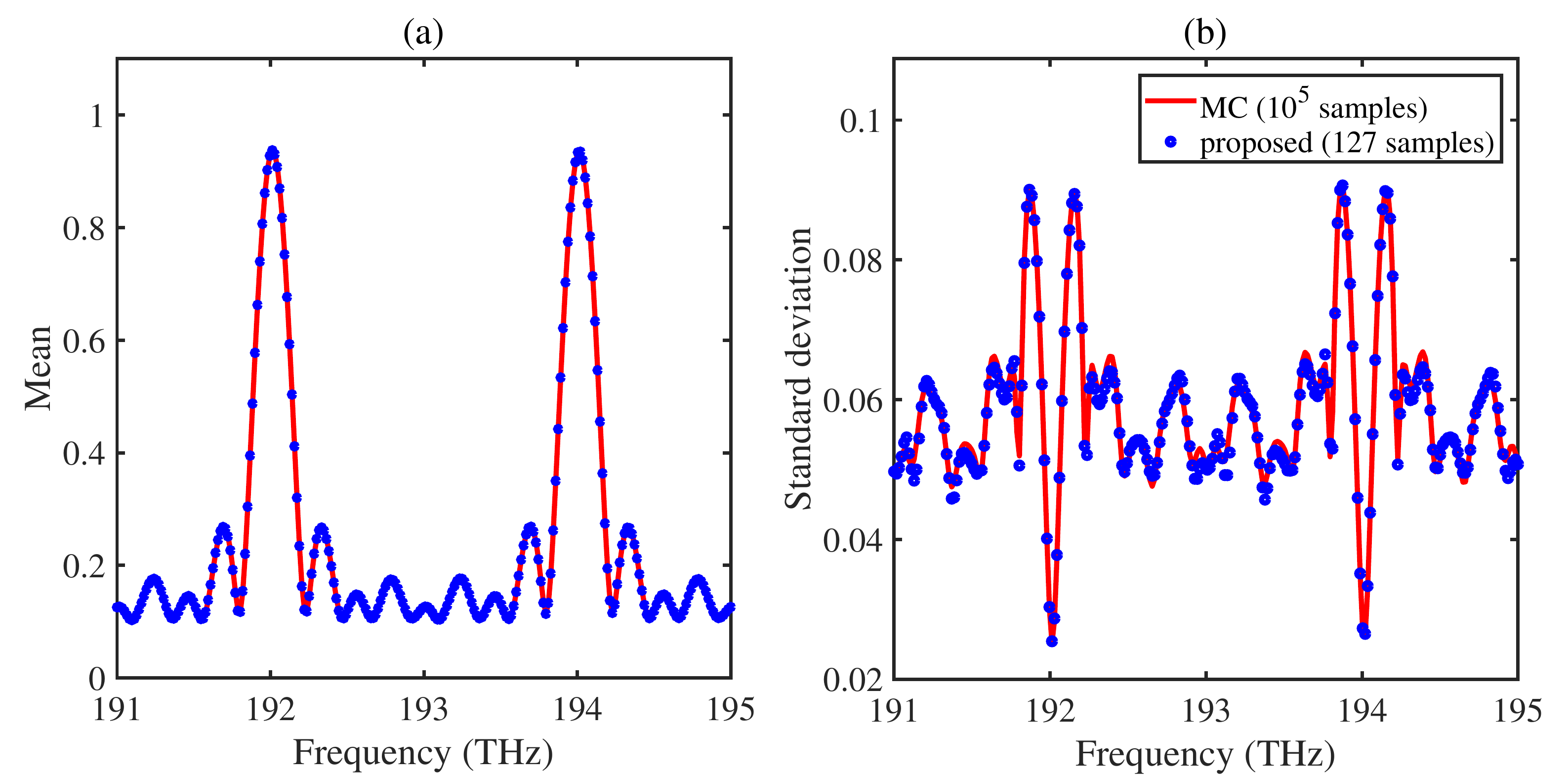}
\caption{Numerical results of the AWG with non-Gaussian correlated uncertainties in radius $R_1, R_2$ and the waveguide array lengths of $L_1,\ldots,L_9$. (a) mean value of the transmission rate; (b) standard deviation of the transmission rate obtained by our proposed method and Monte Carlo (MC).}
	\label{fig:AWG_results}
\end{figure}

We aim to build a 2nd-order stochastic model to approximate the transmission rates. A Gaussian-mixture model is used to describe the geometric uncertainties:
\begin{equation}
\vecpar = \vecpar_0+\Delta\vecpar,\text{ where }\Delta\vecpar\sim \frac12\mathcal{N}(\boldsymbol{\mu}_1,\boldsymbol{\Sigma}_1)+\frac12\mathcal{N}(\boldsymbol{\mu}_2,\boldsymbol{\Sigma}_2).\nonumber
\end{equation}
For the radius of the star couplers,  we set the mean values as
$\boldsymbol{\mu}_1=-\boldsymbol{\mu}_2=29.8\times \mat{1}$ $\mu$m. For the waveguide array lengths, we set $\boldsymbol{\mu}_1=-\boldsymbol{\mu}_2=0.05\times \mat{1}$ $\mu $m. The covariance matrices are block diagonal positive definite.

We compare the computed mean value and standard deviation of our method with that from Monte Carlo in Fig.~\ref{fig:AWG_results}. Using only 127 simulation samples, our method is able to achieve the similar accuracy with $10^5$ Monte Carlo samples.
Fig.~\ref{fig:AWG_pdf} further shows the probability density functions of the transmission rates at two frequency points $f=191.9478$ THz and $f=192.3494$ THz.

\begin{figure}[t]
	\centering \includegraphics[width=90mm,height=1.8in]{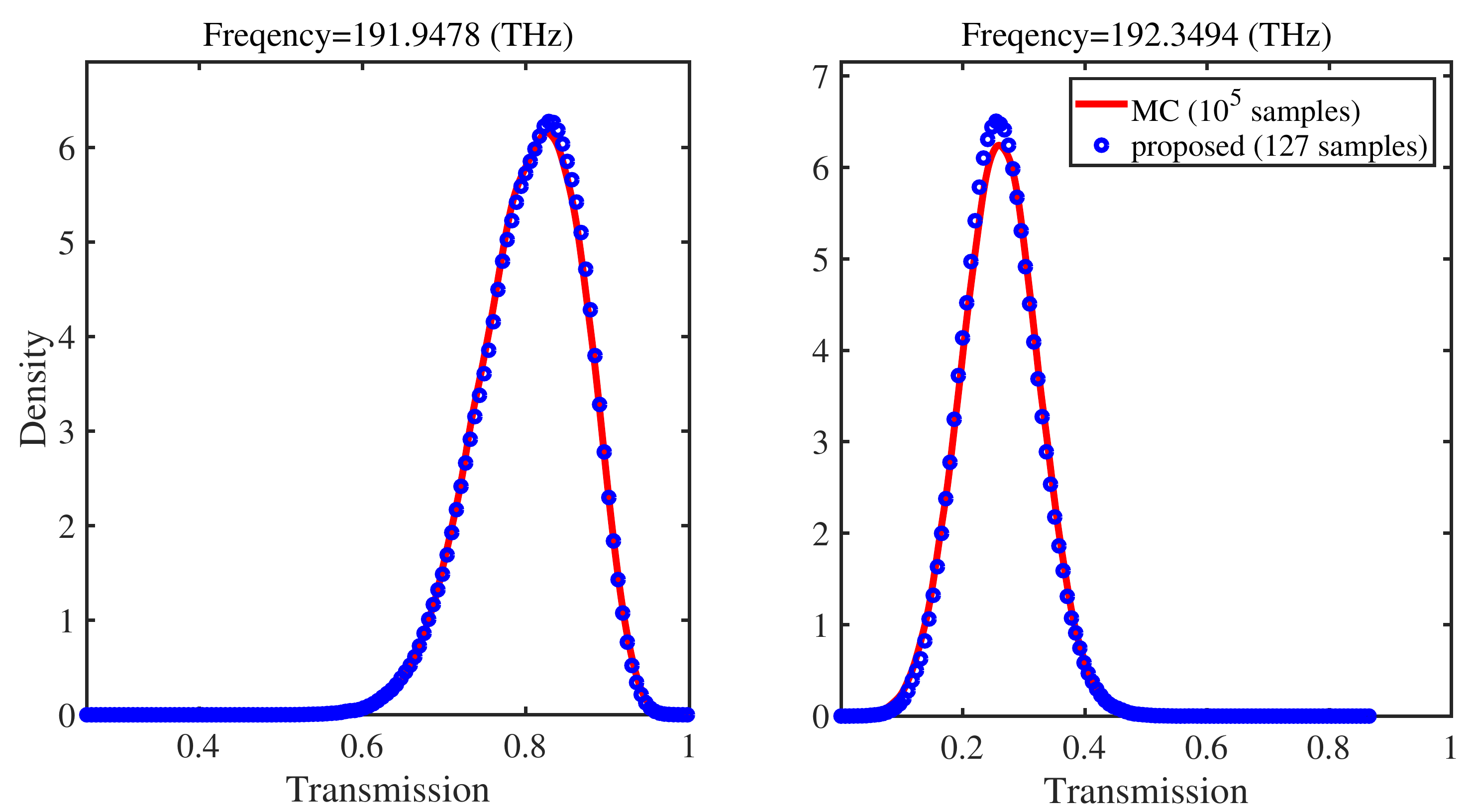}
\caption{Probability density functions of the transmission rates at two frequency points $f=191.9478$ THz and $f=192.3494$ THz obtained by our proposed method and Monte Carlo (MC)}
	\label{fig:AWG_pdf}
\end{figure}

\subsection{Practical Number of Quadrature Samples}
\label{sec:result_num}
Finally, Table~\ref{tab:sample} shows the number of quadrature samples used by our approach in all numerical  experiments. The lower and upper bounds of the number of samples from Theorem~\ref{thm:Mbound} are listed in the last two columns. Clearly, in most cases, the practical number of samples is very close to the lower bound. When the order of basis function is very high, the obtained number of quadrature samples may occasionally becomes close to the upper bound. This is because the following reason: when $p$ is very large, the objective function in \eqref{equ:NLS} is a polynomial function of extremely high order (i.e., $4p$), and the coordinate descent solver becomes hard to converge. We expect that the number of quadrature samples will also be close to the theoretical lower bound even for very large $p$, if a better nonlinear optimization solver is developed in the future.

\begin{table}
\centering
\caption{The number of quadrature samples used in all experiments. Here, $p$ denotes the maximal order of basis functions, $d$ is the number of random parameters.}
\label{tab:sample}
 \begin{tabular}{|c|cc|ccc|}
 %|c@{\hspace{1.0mm}}|@{\hspace{1.0mm}}cc@{\hspace{1.0mm}}|c@{\hspace{1.8mm}}c@{\hspace{1.8mm}}c|}
\hline
&& & \multicolumn{3}{c|}{Proposed}  \\
\hline
Benchmarks& $p$& $d$& \# samples & lower bound &upper bound\\ \thickhline
\multirow{5}{*}{Synthetic} & 1& 2& 3&3&6\\
 & 2& 2&6 &6 &15\\
 & 3& 2& 10&10&28\\
 & 4& 2& 17&15&45\\
 & 5& 2& 66&21&66\\
 \hline
CMOS ring & 2& 6&33&28&210\\ \hline
\multirow{2}{*}{Optical filter} &2& 4&16 & 15& 70   \\
 &2& 11&139 &78 &1365  \\
\hline
AWG & 2& 11&127&78 &1365\\ \hline
\end{tabular}
\end{table}

% , and can achieve 700$\times$ to $6000\times$ speed up compared with Monte Carlo approach.

\section{Conclusion and Remarks}
This paper has investigated a long-standing research challenge: how can we handle non-Gaussian correlated uncertainties by stochastic spectral methods? We have proposed several theories and algorithms to overcome this challenge and have tested them by various benchmarks. Specifically, we have proposed a set of orthonormal basis functions that work extremely well for non-Gaussian correlated process variations which are beyond the capability of the existing well-known generalized polynomial-chaos theory. We have presented an optimization approach to calculate the quadrature nodes and weights required in the projection step. We have also provided some rigorous theoretical results regarding the required number of quadrature samples and the error bound of our framework. Our method has demonstrated a nearly exponential convergence rate on a smooth synthetic example. It has also achieved 700$\times$ to 6000$\times$ speedup than Monte Carlo on several practical design benchmarks, including a CMOS electronic ring oscillator, an optical filter built with 3-stage photonic ring resonators and an arrayed waiveguide grating.

We have two final remarks:
\begin{itemize}
\item Based on our theoretical analysis, we conclude that as long as the stochastic unknown output is smooth enough, and if the the optimization solver in our quadrature rule has a small error, both the numerical integration and approximation error will be very small, leading to highly accurate results in our stochastic collocation framework.

\item It remains an open problem to determinate the required minimum number of quadrature nodes. Our numerical experiments show an excellent heuristic result: the practical number of quadrature nodes used in our framework is almost always close to the theoretical lower bound.

% \item The method proposed in this paper can only deal with low-dimensional uncertainty quantification problems.
% For higher dimensional case, one can use other methods such as the compressive sensing approach proposed in \cite{Cui2018}.
\end{itemize}

% This paper has presented a big-data approach for solving the challenging high-dimensional uncertainty quantification problem. Our key idea is to estimate the high-dimensional simulation data array from an extremely small subset of its samples. This idea has been described as a tensor-recovery model with low-rank and sparse constraints. Simulation results on integrated circuits and MEMS show that our algorithm can be easily applied to problems with over $50$ random parameters. Instead of using a huge number of (e.g., about $10^{27}$) quadrature samples, our algorithm requires only several hundreds which is even much smaller than the number of basis functions.

% but journal paper do

\section*{Acknowledgment}
The authors would like to thank the anonymous reviewers for their detailed comments.
We also appreciate Allen Sadun, Kaiqi Zhag and Kaikai Liu for their helpful discussions on the benchmarks and on Lumerical interconnect, and thank Max Gershman for his help on some of the code implementation.

\appendices
%% appendices A
\section{Proof of Theorem \ref{thm:Mbound}}
\label{append:Mbound}

We show the lower bound and upper bound of the number of quadrature points required to achieve $2p$-th-order accuracy are $N_p$ and $N_{2p}$, respectively.

Firstly, according to Appendix~\ref{app:lem1}, \eqref{equ:prodInt} holds if the quadrature points and weights satisfy \eqref{equ:nmint1}.
As a result, we have
\begin{equation}
\mat{Q}\text{diag}(\mat{w})\mat{Q}^T=\mat{I}_{N_p},
\end{equation}
where $\mat{Q}\in \mathbb{R}^{N_p\times M}$ with each element $\mat{Q}_{ij}=\multiGPC_i(\vecpar_j)$, and $\mat{I}_{N_p}$ is an $N_p$-by-$N_p$ identity matrix.
Because the right-hand side is full rank, $\mat{Q}$ has a full row rank and thus $M\ge N_p$.

We further notice that the first row of \eqref{equ:nmint1} is  $\sum_{k=1}^Mw_k=1$, therefore  \eqref{equ:nmint1} can be rewritten as
\begin{equation}\label{equ:conveccomb}
\mat{Q}_1\mat{w}=\mat{0}_{N_{2p}-1},\ \sum_{k=1}^Mw_k=1, \ \mat{w}\ge0,
\end{equation}
where $\mat{Q}_1\in\mathbb{R}^{(N_{2p}-1)\times M}$ consists of the last $N_{2p}-1$ rows of $\mat{Q}$, and $\mat{0}_{N_{2p}-1}\in\mathbb{R}^{N_{2p}-1}$ is a zero vector.
According to the Carath\'eodory's
Theorem~\cite{barany2012notes}, because $\mat{0}_{N_{2p}}$ lies in the convex hull formed by the column vectors of $\mat{Q}_1$, it can be written as the convex combination of not more than $N_{2p}$ column vectors. In other words, there exists a matrix $\hat{\mat{Q}}_1$ formed by only $N_{2p}$ columns of $\mat{Q}_1$ such that \eqref{equ:conveccomb} still holds if we replace $\mat{Q}_1$ with $\hat{\mat{Q}}_1$ and change the length of $\mat{w}$ accordingly. Vector $\mat{0}_{N_{2p}-1}$ being in the convex hull of $\mat{Q}_1$ is a natural result of our numerical quadrature rule defined on the selected basis functions, therefore there exists $M\le N_{2p}$.

% Intuitively, $\mat{Q}_1$ have $N_{2p}-1$ rows, then $\text{rank}(\mat{Q}_1)\le N_{2p}-1$. Given any $N_{2p}-1$ sample nodes that constructs the basis matrix as $\bar{\mat{Q}}_1\in\mathbb{R}^{(N_{2p}-1)\times (N_{2p}-1)}$, if the $(N_{2p})$-th samples node is computed such that the basis vector    $\mat{q}=-\bar{\mat{Q}}_1\mat{1}$, then for $\mat{Q}_1=[\bar{\mat{Q}}_1, \mat{q}]$, there is $\mat{Q}_1\mat{1}=\mat{0}$.
% Hence, \eqref{equ:conveccomb} holds for $\mat{w}=\frac{1}{N_{2p}}\mat{1}$.

\textbf{Remark}  In the above proof, we show that by Carath\'eodory's Theorem, there exists $N_{2p}$ quadrature nodes and weights such that \eqref{equ:conveccomb} is true. In general, we do not know how to choose the $N_{2p}$ sample nodes and weights {\it a priori}.
However, our optimization solver can automatically calculate these quadrature nodes and weights. On the contrary, the linear programming approach in \cite{ryu2015extensions} needs to prescribe the sampling nodes and only calculate the weights, and it can not guarantee the conditions in \eqref{equ:conveccomb}.

%%%%%% appendix A
\section{Proof of Theorem~\ref{thm:exactrecovery} }
 \label{app:lem1}

In order to show the exact recovery of $y(\vecpar)\in\ten{S}_p$, we need to prove that
\begin{equation}\label{equ:recv}
c_{\basisInd}=\tilde{c}_{\basisInd}, \ \forall \ |\basisInd|\le p.
\end{equation}
Here $\tilde{c}_{\basisInd}$ is obtained by the following numerical scheme:
\begin{align*}
\tilde{c}_{\basisInd} =& \sum_{k=1}^M y(\vecpar_k)\multiGPC_{\basisInd}(\vecpar_k)w_k = \sum_{k=1}^M \sum_{|\boldsymbol{\beta}|=0}^pc_{\boldsymbol{\beta}} \multiGPC_{\boldsymbol{\beta}}(\vecpar_k)\multiGPC_{\basisInd}(\vecpar_k)w_k\\
=&\sum_{|\boldsymbol{\beta}|=0}^p c_{ \boldsymbol{\beta}} \left(\sum_{k=1}^M  \multiGPC_{\boldsymbol{\beta}}(\vecpar_k)\multiGPC_{\basisInd}(\vecpar_k)w_k \right).
\end{align*}
A sufficient condition of   \eqref{equ:recv} is
\begin{equation}\label{equ:prodInt}
\sum_{k=1}^M  \multiGPC_{\boldsymbol{\beta}}(\vecpar_k)\multiGPC_{\basisInd}(\vecpar_k)w_k=\delta_{\basisInd,\boldsymbol{\beta}}.
\end{equation}
In fact, the left-hand side of \eqref{equ:prodInt} is the numerical approximation for the integral $\mathbb{E}[\multiGPC_{\boldsymbol{\beta}}(\vecpar)\multiGPC_{\basisInd}(\vecpar)]$, which is guaranteed to be exact if we have a quadrature rule that can exactly evaluate the integration of every basis function bounded by order $2p$. In other words, \eqref{equ:nmint1} is a sufficient condition for \eqref{equ:prodInt}.%This completes the proof.

%%%%% appendix B
\section{Proof of Theorem \ref{thm:interr}}
 \label{app:thm1}
Before the detailed proof, we first introduce the H\"older's inequality \cite{kuttler2007introduction} that will be used in our theoretical analysis.
\begin{itemize}
\item H\"older's inequality for the Euclidean vector space: for all vectors $\mat{x}, \mat{y}\in\mathbb{R}^n$ and $q_1,q_2 \in [1,+\infty]$ with $\frac1{q_1} + \frac1{q_2} =1$,
\begin{equation}\label{equ:HolderVec1}
\left|\sum_{i=1}^nx_iy_i\right|\le \sum_{i=1}^n|x_iy_i|\le\|\mat{x}\|_{q_1}\|\mat{y}\|_{q_2}.
\end{equation}
For the special case ${q_1}=1$ and ${q_2}=+\infty$, there is
\begin{equation}\label{equ:HolderVec2}
\left|\sum_{i=1}^nx_iy_i\right|\le \sum_{i=1}^n|x_iy_i| \le \|\mat{x}\|_1\|\mat{y}\|_{\infty}.
\end{equation}

\item H\"older's inequality in the probability space:
for all measurable functions $f(\vecpar)$ and $g(\vecpar)$ and $q_1,q_2 \in [1,+\infty]$ with $\frac1{q_1} + \frac1{q_2} =1$:
\begin{equation}\label{equ:HolderProb1}
\mathbb{E}[|f(\vecpar)g(\vecpar)|]\le \|f(\vecpar)\|_{q_1}\|g(\vecpar)\|_{q_2}.
%(\mathbb{E}[|f(\vecpar)|^p])^{\frac1p} (\mathbb{E}[|g(\vecpar)|^q])^{\frac1q}.
\end{equation}
For the special case $g(\vecpar)\equiv 1$ and ${q_1}={q_2}=2$, there is
\begin{equation}\label{equ:HolderProb2}
\|f(\vecpar)\|_1=\mathbb{E}[|f(\vecpar)|]\le (\mathbb{E}[|f(\vecpar)|^2])^{\frac12}=\|f(\vecpar)\|_2.
\end{equation}
\end{itemize}

Now we start to prove Theorem~\ref{thm:interr}.
According to the definition $y_p(\vecpar)=\sum_{|\basisInd|=0}^p c_{\basisInd}\multiGPC_{\basisInd}(\vecpar)$ and $c_{\basisInd}=\mathbb{E}[y(\vecpar)\multiGPC_{\basisInd}(\vecpar)]$, we have
\begin{equation}\label{equ:Ecc}
\mathbb{E}[y(\vecpar)\multiGPC_j(\vecpar)]=\mathbb{E}[y_p(\vecpar)\multiGPC_j(\vecpar)]=c_j, \forall j=1,\ldots,N_p.
\end{equation}
We consider $j=1$ and $\multiGPC_1(\vecpar)=1$, then the above equation indicates $\mathbb{E}[y(\vecpar)] = \mathbb{E}[y_p(\vecpar)]=c_0$. Based on this observation, we can estimate the difference between $\mathbb{E}[y(\vecpar)]$ and $\mathbb{I}[y(\vecpar)]$:
\begin{align}
\nonumber&\left|\mathbb{E}[y(\vecpar)]-\mathbb{I}[y(\vecpar)]\right| =\left|\mathbb{E}[y_p(\vecpar)]-\mathbb{I}[y(\vecpar)]\right| \\
\le&\underbrace{\left|\mathbb{E}[y_p(\vecpar)]-\mathbb{I}[y_p(\vecpar)]\right|}_{\text{(a)} }+\underbrace{\left|\mathbb{I}[y_p(\vecpar)]-\mathbb{I}[y(\vecpar)]\right|}_{\text{(b)}}.
\label{equ:Euppbd}
\end{align}

Item (a) arises from the error of our numerical quadrature:
\begin{align}
\nonumber
(a)=&\left|\mathbb{E}[y_p(\vecpar)]-\mathbb{I}[y_p(\vecpar)]\right| =\left|\sum_{j=1}^{N_p} c_j\left(\mathbb{E}[\multiGPC_j(\vecpar)]-\mathbb{I}[\multiGPC_j(\vecpar)]\right)\right|\\
 \le & \|\mat{c}\|_{\infty}\|\Phi(\bar{\vecpar})\mat{w}-\mat{e}_1\|_1
\le L\epsilon. \label{equ:errthm1a}
\end{align}
The first inequality results from the H\"older's inequality \eqref{equ:HolderVec2}. The second inequality follows from $\|\Phi(\bar{\vecpar})\mat{w}-\mat{e}_1\|_1\le\epsilon$ in \eqref{equ:numerr}, and we have $\|\mat{c}\|_{\infty}\le L$ because
\begin{equation}
\label{equ:bdc}
|c_j|=|\mathbb{E}[y(\vecpar)\multiGPC_j(\vecpar)]|\le\|y(\vecpar)\|_2\|\multiGPC_j(\vecpar)\|_2\le L\|\multiGPC_j(\vecpar)\|_2=L.\nonumber
\end{equation}

Item (b) is due to the projection  error
\begin{align}
\nonumber (b)=&\left|\mathbb{I}[y_p(\vecpar)-\mathbb{I}[y(\vecpar)]\right|\\
\le& W \|y(\vecpar)-y_p(\vecpar)\|_1 \le W \|y(\vecpar)-y_p(\vecpar)\|_2 \le W\delta.\label{equ:errthm1b}
\end{align}
The first inequality follows from that the operator $\mathbb{I}$ is bounded by $W$ in \eqref{equ:Ibd}. The second inequality results from the H\"older's inequality \eqref{equ:HolderProb2}. The
last inequality follows from our assumption $\|y(\vecpar)-y_p(\vecpar)\|_2\le \delta$ in \eqref{equ:apperr}.

% $|\mathbb{E}[\multiGPC_1(\vecpar)]-\sum_{k=1}^M\multiGPC_1(\vecpar_k)w_k|\le \delta$. Submitting $\multiGPC_1(\vecpar)=1$ into this equation, we have $|1-\sum_{k=1}^Mw_k|\le\delta$. That is to say, $\sum_{k=1}^Mw_k\le 1+\delta$.

Combing \eqref{equ:Euppbd}, \eqref{equ:errthm1a} and \eqref{equ:errthm1b}, we have
\begin{equation}
\left|\mathbb{E}[y(\vecpar)]-\mathbb{I}[y(\vecpar)]\right|\le L\epsilon+W\delta.
\end{equation}
The proof of Theorem \ref{thm:interr} is complete.

\section{Proof of Theorem \ref{thm:apprerr}}
 \label{app:thm2}

The total error of our stochastic collocation algorithm can be bounded by two terms:
\begin{equation}\label{equ:f_fbar}
\|y(\vecpar)-\tilde y(\vecpar)\|_2\le\|y(\vecpar)-y_p(\vecpar)\|_2+\|y_p(\vecpar)-\tilde y(\vecpar)\|_2. \nonumber
\end{equation}
Based on Assumption 2, the first item is upper bounded by $\delta$. We only need to estimate the second term.
In fact,
\begin{equation}\label{equ:th2b}
\|y_p(\vecpar)-\tilde y(\vecpar)\|_2=\|\sum_{j=1}^{N_p} (c_j-\tilde c_j)\multiGPC_j(\vecpar)\|_2=\sqrt{\sum_{j=1}^{N_p} (c_j-\tilde{c}_j)^2}, \nonumber
\end{equation}
where the last equality follows the fact that the chosen basis functions are orthogonal and normalized. Furthermore,
\begin{align}
\nonumber\left|c_{j}-\tilde c_j\right| =&\left|\mathbb{E}[y_p(\vecpar)\multiGPC_j(\vecpar)]-\mathbb{I}[y(\vecpar)
\multiGPC_j(\vecpar)]\right|\\
\le&\underbrace{\left|\mathbb{E}[y_p(\vecpar)\multiGPC_j(\vecpar)]-\mathbb{I}[y_p(\vecpar)
\multiGPC_j(\vecpar)]\right|}_{(a)}\label{equ:1st}
\\
&+\underbrace{
\left|\mathbb{I}[(y_p(\vecpar)-y(\vecpar))\multiGPC_j(\vecpar)]\right|}_{(b)}.\label{equ:2nd}
\end{align}

Both $y_p(\vecpar)$ and $\multiGPC_j(\vecpar)$ are polynomials bounded by order $p$, so their product is a polynomial bounded by order $2p$, i.e., $y_p(\vecpar)
\multiGPC_j(\vecpar)\in\mathcal{S}_{2p}$. There exists an expansion $y_p(\vecpar)
\multiGPC_j(\vecpar)=\sum_{l=1}^{N_{2p}}a_l\multiGPC_l(\vecpar)$ and an upper bound for term (a):
\begin{align}
\nonumber(a)=&\left|\mathbb{E}[y_p(\vecpar)\multiGPC_j(\vecpar)]-\mathbb{I}[y_p(\vecpar)
\multiGPC_j(\vecpar)]\right|\\
\nonumber= &\left|\sum_{l=1}^{N_{2p}}a_l \left(\mathbb{E}[\multiGPC_l(\vecpar)]-\mathbb{I}[
\multiGPC_l(\vecpar)]\right)\right|\\
\le & \|\mat{a}\|_{\infty}\|\Phi(\bar{\vecpar})\mat{w}-\mat{e}_1\|_1 \le LT\epsilon.\label{equ:bd1st}
\end{align}
The first inequality is due to \eqref{equ:HolderVec2}, and the last inequality follows from
\begin{align}
 a_l=\mathbb{E}[y_p(\vecpar)\multiGPC_j(\vecpar)\multiGPC_l(\vecpar)]\le \|y_p(\vecpar)\|_2 \|\multiGPC_j(\vecpar)\multiGPC_l(\vecpar)\|_2\le  LT
\label{equ:bounda}
\end{align}
where % key trouble in the proof
$T=\max_{j,l=1,\ldots,N_{2p}}\|\multiGPC_j(\vecpar)\multiGPC_l(\vecpar)\|_2$.

We can also find an upper bound for term $(b)$ in \eqref{equ:2nd}:
\begin{align}
\nonumber (b)=& \left|\mathbb{I}[(y_p(\vecpar)-y(\vecpar))\multiGPC_j(\vecpar)]\right|\\
\nonumber\le &  W\|(y_p(\vecpar)-y(\vecpar))\multiGPC_j(\vecpar)\|_1\\
\nonumber \le &W \|(y_p(\vecpar)-y(\vecpar))\|_2\|\multiGPC_j(\vecpar)\|_2  \\
 = &W \|(y_p(\vecpar)-y(\vecpar))\|_2  \le W\delta. \label{equ:bd2th2}
\end{align}
Combing \eqref{equ:1st}, \eqref{equ:2nd}, \eqref{equ:bd1st} and \eqref{equ:bd2th2}, we have $|c_j-\tilde{c}_j|\le LT\epsilon + W\delta$, and thus $\|y_p(\vecpar)-\tilde y(\vecpar)\|_2\le N_p(LT\epsilon + W\delta)$. Noting that $\|y(\vecpar)-y_p(\vecpar)\|_2\le \delta$, we finally have
\begin{equation}
\|y(\vecpar)-\tilde{y}(\vecpar)\|_2\le \delta + N_p(LT\epsilon + W\delta).
\end{equation}
This completes the proof of Theorem  \ref{thm:apprerr}.

\textbf{Remark} To show $a_l$ is bounded in \eqref{equ:bounda} is equivalent to show  $ \ten{S}_{2p}$ is complete under the Minkowski sum, i.e.,
\begin{equation}\label{equ:mikis}
\ten{S}_p\oplus\ten{S}_p\subset\ten{S}_{2p}.
\end{equation}
In other words, if $ p_1(\vecpar), p_2(\vecpar)\in \ten{S}_{p}$, then $ p_1(\vecpar)p_2(\vecpar)\in \ten{S}_{2p}$.
Intuitively, this is true because the product of two $p$-th order polynomial is a polynomial bounded by order $2p$. A sufficient condition for \eqref{equ:mikis} is that $\|\multiGPC_j(\vecpar)\multiGPC_l(\vecpar)\|_2$ is bounded. In real applications, most widely used distributions include Gaussian, Gaussian mixture distribution, or a distribution on a bounded domain can guarantee that the high-order moments are bounded. As a result, \eqref{equ:mikis} holds in most cases.
But there exists some rare density functions whose high-order moments are not necessarily bounded, such as the log norm distribution.
In this rare case, the error analysis in Theorem \ref{thm:apprerr} may not hold.

\section{Proof for Lemma \ref{lemma:coverr}}
\label{app:coverr}

In order to upper bound $\|\mat{V}-\mat{I}_{N_p}\|_F$, we consider the  error for   each element $\mathbb{E}[\multiGPC_i(\vecpar)\multiGPC_j(\vecpar)]-\mathbb{I} [\multiGPC_i(\vecpar)\multiGPC_j(\vecpar)]$.
We can have an expansion $\multiGPC_i(\vecpar)\multiGPC_j(\vecpar)=\sum_{l=1}^{N_{2p}} a_l\multiGPC_l(\vecpar)$,
then
\begin{align*}
&\left|\mathbb{E}[\multiGPC_i(\vecpar)\multiGPC_j(\vecpar)]-\mathbb{I} [\multiGPC_i(\vecpar)\multiGPC_j(\vecpar)]\right| \\
=&\left|\sum_{l=1}^{N_{2p}} a_l\left(\mathbb{E}[\multiGPC_l(\vecpar)]-\mathbb{I} [\multiGPC_l(\vecpar)]\right)\right| \\
\le
&\|\mat{a}\|_2\|\Phi(\bar{\vecpar})\mat{w}-\mat{e}_1\|_2.
\end{align*}
Because $
\|\mat{a}\|_2^2=\|\multiGPC_i(\vecpar)\multiGPC_j(\vecpar)\|_2^2\le T^2$ and
\begin{align*}
\|\Phi(\bar{\vecpar})\mat{w}-\mat{e}_1\|_2\le \|\Phi(\bar{\vecpar})\mat{w}-\mat{e}_1\|_1\le \epsilon,
\end{align*}
we have
\begin{align*}
\left|\mathbb{E}[\multiGPC_i(\vecpar)\multiGPC_j(\vecpar)]-\mathbb{I} [\multiGPC_i(\vecpar)\multiGPC_j(\vecpar)]\right|\le T\epsilon,
\end{align*}
and further obtain $\|\mat{V}-\mat{I}_{N_p}\|_F\le N_pT\epsilon$.

 \bibliographystyle{IEEEtran}
\bibliography{nonGaussianSC}

 \begin{IEEEbiography}
  [{\includegraphics[width=1in,height=1.25in,clip,keepaspectratio]{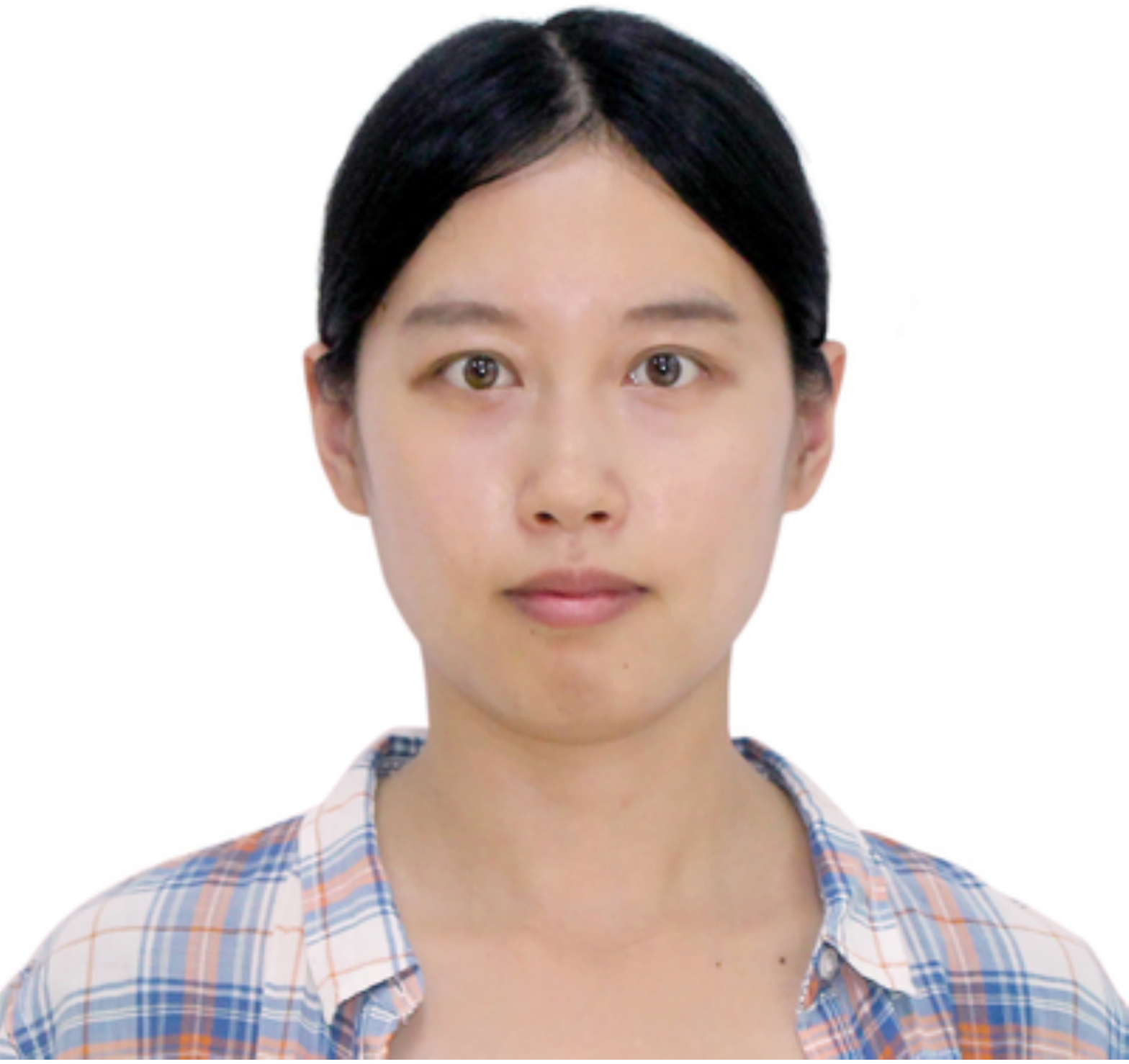}}]{Chunfeng Cui} received the Ph.D. degree in computational mathematics from Chinese Academy of Sciences, Beijing, China, in 2016 with a specialization in numerical optimization. From 2016 to 2017, she was a Postdoctoral Fellow at City University of Hong Kong, Hong Kong. In 2017 She joined the Department of Electrical and Computer Engineering at University of California Santa Barbara as a Postdoctoral Scholar.
From 2011 her research activity is mainly focused in the areas of tensor analysis and its applications. She has been working on numerical optimization algorithms for tensor problems, and its applications for machine learning and for uncertainty quantification of nano-scale chip design. She received the Best Paper Award of the IEEE EPEPS 2018.
\end{IEEEbiography}

\begin{IEEEbiography}
 [{\includegraphics[width=1in,height=1.25in,clip,keepaspectratio]{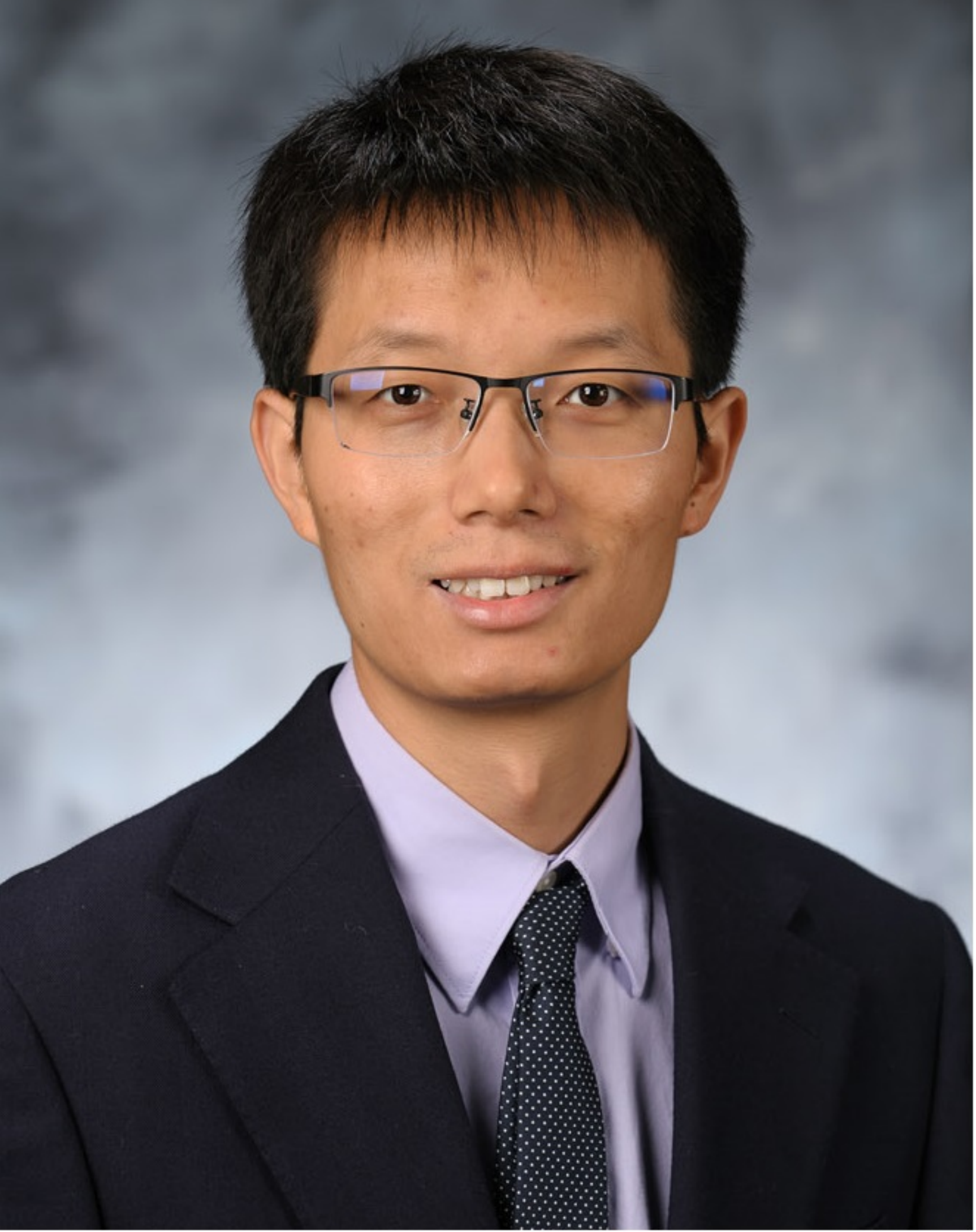}}]{Zheng Zhang} (M'15) received his Ph.D degree in Electrical Engineering and Computer Science from the Massachusetts Institute of Technology (MIT), Cambridge, MA, in 2015. He is an Assistant Professor of Electrical and Computer Engineering with the University of California at Santa Barbara (UCSB), CA. His research interests include uncertainty quantification with applications to the design automation of multi-domain systems (e.g., nano-scale electronics, integrated photonics, and autonomous systems), and tensor computational methods for high-dimensional data analytics. His industrial experiences include Coventor Inc. and Maxim-IC; academic visiting experiences include UC San Diego, Brown University and Politechnico di Milano; government lab experiences include Argonne National Labs.

Dr. Zhang received the Best Paper Award of IEEE Transactions on Computer-Aided Design of Integrated Circuits and Systems in 2014, the Best Paper Award of IEEE Transactions on Components, Packaging and Manufacturing Technology in 2018, two Best Paper Awards (IEEE EPEPS 2018 and IEEE SPI 2016) and three additional Best Paper Nominations (CICC 2014, ICCAD 2011 and ASP-DAC 2011) at international conferences. His PhD dissertation was recognized by the ACM SIGDA Outstanding Ph.D Dissertation Award in Electronic Design Automation in 2016, and by the Doctoral Dissertation Seminar Award (i.e., Best Thesis Award) from the Microsystems Technology Laboratory of MIT in 2015. He was a recipient of the Li Ka-Shing Prize from the University of Hong Kong in 2011.
\end{IEEEbiography}

\end{document}